\documentclass{amsart}
\usepackage[usenames,dvipsnames]{xcolor}
\usepackage{multirow}
\usepackage[all,dvips]{xy}
\usepackage[dvips]{graphicx}
\usepackage[T1]{fontenc}
\usepackage{longtable}
\usepackage{amssymb,amsmath}
\usepackage{pdflscape}
\usepackage{array}

\usepackage{todonotes} 

\DeclareMathOperator{\tors}{tors}

\newtheorem{theorem}{Theorem}
\newtheorem{corollary}[theorem]{Corollary}

\newtheorem{lemma}[theorem]{Lemma}
\newtheorem{proposition}[theorem]{Proposition}

\newcommand{\cC}{{\mathcal{C}}}
\newcommand{\cO}{{\mathcal{O}}}
\newcommand{\Q}{{\mathbb{Q}}}

\newcommand{\Z}{{\mathbb{Z}}}

\title[Torsion of rational elliptic curves over quadratic fields II]{Torsion of rational elliptic curves \\ over quadratic fields II}

\author{Enrique Gonz\'alez--Jim\'enez}
\address{Universidad Aut{\'o}noma de Madrid, Departamento de Matem{\'a}ticas, Madrid, Spain}
\email{enrique.gonzalez.jimenez@uam.es}
\urladdr{http://www.uam.es/enrique.gonzalez.jimenez}
\author{Jos\'e M. Tornero}
\address{Departamento de \'Algebra and IMUS, Universidad de Sevilla. P.O. 1160. 41080 Sevilla, Spain.}
\email{tornero@us.es}
\thanks{The first author was partially  supported by the grant MTM2012--35849. The second author was partially supported by the grant FQM--218 and P12--FQM--2696.}

\begin{document}

\begin{abstract}
Let $E$ be an elliptic curve defined over $\Q$ and let $G = E(\Q)_{\tors}$ be the associated torsion group. In a previous paper, the authors studied, for a given $G$, which possible groups $G \leq H$ could appear such that $H=E(K)_{\tors}$, for $[K:\Q]=2$. In the present paper, we go further in this study and compute, under this assumption and for every such $G$, all the possible situations where $G \neq H$. The result is optimal, as we also display examples for every situation we state as possible. As a consequence, the maximum number of quadratic number fields $K$ such that $E(\Q)_{\tors}\neq E(K)_{\tors}$ is easily obtained.
\end{abstract}

\maketitle

\section{Introduction}

Let $E$ be an elliptic curve defined over a number field $L$. The Mordell-Weil Theorem states that the set of $L$--rational points, $E(L)$, is a finitely generated abelian group. So it can be written as $E(L) = E(L)_{\tors} \oplus \Z^r$, for some non-negative integer $r$ (called the rank of $E(L)$) and some finite torsion subgroup $E(L)_{\tors}$. It is well known that there exist two positive integers $n,m$ such that $n|m$ and $E(L)_{\tors}$ is isomorphic to $\cC_n \times \cC_m$, where $\cC_n$ is the cyclic group of order $n$ \cite{Silverman2009}.

Through this paper, we will often write $G=H$ (respectively $G \leq H$ or $G < H$) for the fact that $G$ is {\em isomorphic} to $H$ (repectively, isomorphic to a subgroup of $H$ or to a proper subgroup of $H$) without further detail on the precise isomorphism. 

We define some useful sets for the sequel: 
\begin{itemize}
\item Let $\Phi(d)$ be the set of possible groups that can appear as the torsion subgroup of an elliptic curve defined over a certain number field $L$ of degree $d$. 
\item Let $\Phi_\Q(d)$ be the set of possible groups that can appear as the torsion subgroup over a number field of degree $d$, of an elliptic curve $E$ defined over the rationals.
\item Let $G \in \Phi(1)$. We will write $\Phi_\Q(d,G)$ the set of possible groups that can appear as the torsion subgroup over any number field $L$ of degree $d$, of an elliptic curve $E$ defined over the rationals, such that $E(\Q)_{\tors}= G$.
\end{itemize}

Connected to these sets, some known results are:
\begin{itemize}
\item Mazur's landmark papers \cite{Mazur1977,Mazur1978} established that
$$
\Phi(1) = \left\{ \cC_n \; | \; n=1,\dots,10,12 \right\} \cup \left\{ \cC_2 \times \cC_{2m} \; | \; m=1,\dots,4 \right\}.
$$
\item After this, in a long series of papers by Kenku, Momose and Kamienny ending in \cite{Kamienny1992, Kenku-Momose1988}, the quadratic case was given a description:
\begin{eqnarray*}
\Phi(2) &=& \left\{ \cC_n \; | \; n=1,\dots,16,18 \right\} \cup \left\{ \cC_2 \times \cC_{2m} \; | \; m=1,\dots,6 \right\}  \cup \\
&& \left\{ \cC_3 \times \cC_{3r} \; | \; r=1,2 \right\} \cup \left\{ \cC_4 \times \cC_4 \right\}.  
\end{eqnarray*}
\item The sets $\Phi_{\mathbb Q}(d)$ have been completely described by Najman \cite{Najman2012preprint} for $d = 2, 3$:
\begin{eqnarray*}
\Phi_{\mathbb Q}(2) &=& \left\{ \cC_n \; | \; n=1,\dots,10,12,15,16 \right\} \cup \left\{ \cC_2 \times \cC_{2m} \; | \; m=1,\dots,6 \right\} \cup \\
&& \left\{ \cC_3 \times \cC_{3r} \; | \; r=1,2 \right\} \cup \left\{ \cC_4 \times \cC_4 \right\},  \\
\Phi_{\mathbb Q}(3) &=& \left\{ \cC_n \; | \; n=1,\dots,10,12,13,14,18,21 \right\} \cup \left\{ \cC_2 \times \cC_{2m} \; | \; m=1,\dots,4,7 \right\}.
\end{eqnarray*}
\item The work of Fujita \cite{Fujita2005} gave the precise list (building upon previous work of Laska and Lorenz \cite{Laska-Lorenz1985}) of torsion groups over the maximal elementary abelian $2$--extension of $\Q$, of elliptic curves defined over the rationals. The full list of such groups will be denoted by  $\Phi_{\mathbb Q}(2^{\infty})$:
\begin{eqnarray*}
\Phi_{\mathbb Q}(2^{\infty}) &=& \left\{ \cC_n \; | \; n=1,3,5,7,9,15 \right\} \cup \left\{ \cC_2 \times \cC_{2m} \; | \; m=1,...,6,8 \right\} \cup \\
&& \left\{ \cC_3 \times \cC_3 \right\} \cup \left\{ \cC_4 \times \cC_{4r}  \; | \; r=1,\dots,4 \right\}\cup \left\{ \cC_{2s} \times \cC_{2s}  \; | \; s=3,4 \right\}.
\end{eqnarray*}
\item The set $\Phi_{\mathbb Q}(2,G)$, for non--cyclic $G$ was characterized by Kwon \cite{Kwon1997}.
\end{itemize}

Finally, in \cite{Gonzalez-Jimenez-Tornero2014}, we gave a precise description of the set $\Phi_{\mathbb Q}(2,G)$, for all $G \in \Phi(1)$.

\begin{theorem}\label{teo1}
For $G \in \Phi(1)$, the set $\Phi_\Q(2,G)$ is the following:
$$
\begin{array}{|c|c|}
\hline
G & \Phi_\Q \left(2,G \right)\\
\hline
\cC_1 & \left\{ \cC_1\,,\, \cC_3 \,,\, \cC_5\,,\, \cC_7\,,\, \cC_9 \right\} \\
\hline
\cC_2 & \left\{ \cC_2\,,\, \cC_4 \,,\, \cC_6\,,\, \cC_8\,,\, \cC_{10}\,,\, \cC_{12}\,,\, \cC_{16}\,,\, \cC_2 \times \cC_{2}\,,\, \cC_2 \times \cC_{6}\,,\, \cC_2 \times \cC_{10} \right\} \\ 
\hline
\cC_3 & \left\{ \cC_3, \; \cC_{15}, \; \cC_3 \times \cC_3 \right\} \\
\hline
\cC_4 & \left\{ \cC_4 \,,\,\cC_8\,,\, \cC_{12}\,,\, \cC_2 \times \cC_{4}\,,\, \cC_2 \times \cC_{8}\,,\, \cC_2 \times \cC_{12}\,,\, \cC_4 \times \cC_4\right\} \\
\hline
\cC_5 & \left\{ \cC_5, \; \cC_{15} \right\} \\
\hline
\cC_6& \left\{ \cC_6, \; \cC_{12}, \; \cC_2 \times \cC_6, \; \cC_3 \times \cC_6 \right\} \\
\hline
\cC_7 & \left\{ \cC_7 \right\} \\
\hline
\cC_8 & \left\{ \cC_8, \; \cC_{16}, \; \cC_2 \times \cC_8 \right\} \\
\hline
\cC_9 & \left\{ \cC_9 \right\} \\
\hline
\cC_{10} & \left\{ \cC_{10}, \; \cC_2 \times \cC_{10} \right\} \\
\hline
\cC_{12} & \left\{ \cC_{12}, \; \cC_2 \times \cC_{12} \right\} \\
\hline
\cC_2 \times \cC_2 & \left\{ \cC_2 \times \cC_{2}\,,\, \cC_2 \times \cC_{4}\,,\, \cC_2 \times \cC_{6}\,,\, \cC_2 \times \cC_{8}\,,\,  \cC_2 \times \cC_{12}\right\} \\
\hline
\cC_2 \times \cC_4 & \left\{ \cC_2 \times \cC_4, \; \cC_2 \times \cC_8, \; \cC_4 \times \cC_4 \right\} \\
\hline
\cC_2 \times \cC_6 & \left\{ \cC_2 \times \cC_6, \cC_2 \times \cC_{12} \right\} \\
\hline
\cC_2 \times \cC_8 & \left\{ \cC_2 \times \cC_8 \right\} \\
\hline
\end{array}
$$
\end{theorem}

\vspace{.1cm}
Let us fix now some useful notations:
\begin{itemize}
\item We will use letters $L$ and $F$ for generic number fields, whereas $K$ will be reserved for proper quadratic extensions of ${\mathbb Q}$.
\item We will denote by ${\mathbb Q}(2^\infty) = \Q \left( \left\{ \sqrt{m} \, | \, m \in \Z \right\} \right)$, the maximal elementary abelian $2$--extension of $\Q$.
\item Let $E$ be an elliptic curve defined over a number field $L$. Without loss of generality we can assume $E$ is defined by a short Weierstrass form
$$
E: Y^2 = X^3 + AX + B; \quad A,B \in L,
$$
and we will then write,
$$
E(L) = \left\{ (x,y) \in L^2 \; | \; y^2=x^3+Ax+B \right\}\cup \{\cO\},
$$
the set of $L$--rational points of $E$, and $\cO$ its point at infinity.
\item For an elliptic curve $E$, let $\Delta_E$ be, as customary, its discriminant.
\item For an elliptic curve $E$ and an integer $n$, let $E[n]$ be the subgroup of all points whose order is a divisor of $n$ (over $\overline{\Q}$), and let $E(L)[n]$ be the set of points in $E[n]$ with coordinates in $L$, for any number field $L$ (including the case $L=\Q$).
\item Under the same conditions, let $\Q(E[n])$ be the extension generated by all the coordinates of points in $E[n]$.
\item For an elliptic curve $E$ defined over the rationals given by a short Weierstrass equation $E:Y^2=X^3+AX+B$, and a squarefree integer $D$, let $E_D$ denote its quadratic twist. That is, the elliptic curve with the Weierstrass  equation $E_D:DY^2=X^3+AX+B$.
\end{itemize}

Please mind that, in the sequel, for examples and particular curves we will use the Antwerp--Cremona tables and labels \cite{antwerp,cremonaweb}.

Our aim in this paper is to go further than we did in \cite{Gonzalez-Jimenez-Tornero2014}. More precisely, at the end of \cite{Gonzalez-Jimenez-Tornero2014} we posed three questions (named Problems $1$, $2$ and $3$). Problems $1$ and $3$ are generalized in the following question:
\vspace{1mm}

\noindent {\bf Question.--} For a given $G \in \Phi(1)$, let $S = \{ H_1,...,H_n \} \subset \Phi_\Q(2,G)$. Find if there exists a fixed elliptic curve $E$ defined over the rationals and squarefree integers $D_1,...,D_r$ such that:
\begin{itemize}
\item $E(\Q)_{\tors} = G$,
\item $E( \Q ( \sqrt{D_i} ))_{\tors} = H_i$, for $i=1,...,n$,
\item $G=E(K)_{\tors}$ for every other quadratic extension $K/\Q$.
\end{itemize}


We will answer this question, which will imply the solution to Problems $1$ and $3$ in \cite{Gonzalez-Jimenez-Tornero2014} as a direct corollary. 

More precisely, we will prove two main results. First, we will compute explicitly how many quadratic extensions $K/\Q$ one can have with a proper extension of the torsion group for a given curve, depending only on the rational torsion structure. This will be done in the following result:

\begin{theorem}\label{teo2}
Let be $G \in \Phi(1)$ and $H\in \Phi_\Q \left(2,G \right)$ such that $G\ne H$. Then the number $h$ of possible quadratic fields $K$ such that $E(\Q)_{\tors}= G$ and $E(K)_{\tors}= H$ for a fixed rational elliptic curve $E$ is given in the following table:
\\[2mm]

\begin{tabular}{ccc}
\begin{tabular}{|c|c|c|}
\hline
$G$ & $H$ & $h$ \\
\hline
\hline
\multirow{4}{*}{$\cC_1$} & $\cC_3$ & $1\,,\,2$ \\
\cline{2-3}
& $\cC_5$  & \multirow{3}{*}{$1$} \\
\cline{2-2}
& $\cC_7$  & \\
\cline{2-2}
& $\cC_9$  & \\
\hline
\hline
\multirow{9}{*}{$\cC_2$} & $\cC_4$ &  \multirow{3}{*}{$1\,,\,2 $} \\
\cline{2-2}
& $\cC_6$  &  \\
\cline{2-2}
& $\cC_8$  & \\
\cline{2-3}
& $\cC_{10}$  &\multirow{6}{*}{$1$}\\
\cline{2-2}
& $\cC_{12}$  &  \\
\cline{2-2}
& $\cC_{16}$  &\\
\cline{2-2}
& $\cC_{2}\times\cC_{2}$  &\\
\cline{2-2}
& $\cC_{2}\times\cC_{6}$  &\\
\cline{2-2}
& $\cC_{2}\times\cC_{10}$ &\\
\hline
\end{tabular}
&
\begin{tabular}{|c|c|c|}
\hline
$G$ & $H$ & $h$ \\
\hline
\hline
\multirow{2}{*}{$\cC_3$} & $\cC_{15}$ &  \multirow{2}{*}{$1$} \\
\cline{2-2}
& $\cC_{3}\times\cC_{3}$  & \\
\hline
\hline
\multirow{6}{*}{$\cC_4$} & $\cC_8$ & $2$\\
\cline{2-3}
& $\cC_{12}$  &  \multirow{5}{*}{$1$} \\
\cline{2-2}
& $\cC_{2}\times\cC_{4}$  &\\
\cline{2-2}
& $\cC_{2}\times\cC_{8}$  &\\
\cline{2-2}
& $\cC_{2}\times\cC_{12}$  &\\
\cline{2-2}
& $\cC_{4}\times\cC_{4}$  &\\
\hline
\hline
$\cC_{5}$ & $\cC_{15}$  &  $1$\\
\hline
\hline
\multirow{3}{*}{$\cC_6$} & $\cC_{12}$ & $2$\\
\cline{2-3}
& $\cC_{2}\times\cC_{6}$  &  \multirow{2}{*}{$1$} \\
\cline{2-2}
& $\cC_{3}\times\cC_{6}$  &\\
\hline
\end{tabular}
&
\begin{tabular}{|c|c|c|}
\hline
$G$ & $H$ & $h$ \\
\hline
\hline
\multirow{2}{*}{$\cC_8$} & $\cC_{16}$ & $2$\\
\cline{2-3}
& $\cC_{2}\times\cC_{8}$  &  $1$ \\
\hline
\hline
$\cC_{10}$ & $\cC_{2}\times \cC_{10}$  &  $1$\\
\hline
\hline
$\cC_{12}$ & $\cC_{2}\times \cC_{12}$  &  $1$\\
\hline
\hline
\multirow{4}{*}{$\cC_{2}\times \cC_{2}$} & $\cC_{2}\times \cC_{4}$ & $1,\,2,\,3$\\
\cline{2-3}
& $\cC_{2}\times\cC_{6}$  &   \multirow{3}{*}{$1$}  \\
\cline{2-2}
& $\cC_{2}\times\cC_{8}$  &   \\
\cline{2-2}
& $\cC_{2}\times\cC_{12}$  &  \\
\hline
\hline
\multirow{2}{*}{$\cC_{2}\times\cC_4$} & $\cC_{2}\times\cC_{8}$ & $1,\,2$\\
\cline{2-3}
& $\cC_{4}\times\cC_{4}$  &  $1$ \\
\hline
\hline
$\cC_{2}\times\cC_6$ & $\cC_{2}\times\cC_{12}$  &  $1$\\
\hline
\end{tabular}
\end{tabular}
\end{theorem}

Once this is done, we will solve a more delicate problem. We will compute, for a given $G \in \Phi(1)$, all the possibilities for $\Phi_{\mathbb Q}(2,G)$ that actually appear. That is, the full set:
$$
\mathcal{H}_{\Q}(2,G) = \{ S_1,...,S_n \}
$$
satisfying, for all $i=1,...,n$, that 
$$
S_i= \left[ H_1,...,H_m \right]
$$ 
is a list, with $H_j \in \Phi_{\mathbb Q}(2,G) \setminus \{ G \}$, and there exists an elliptic curve $E_i$ defined over ${\mathbb Q}$ such that:
\begin{itemize}
\item $E_i({\mathbb Q})_{\tors} = G$,
\item there are quadratic fields $K_1,...,K_m$ with $E_i ( K_j)_{\tors} = H_j$, for all $j=1,...,m$,
\item $E_i(K)_{\tors} = G$, for any other quadratic extension $K / \Q$.
\end{itemize}

Note that we are admitting the possibility of two (or more) of the $H_j$ being identical. We describe explicitly $\mathcal{H}_{\Q}(2,G)$ in Theorem \ref{teo3}. 

\begin{theorem}\label{teo3}
Let be $G \in \Phi(1)$ such that $\Phi_\Q \left(2,G \right)\ne \{G\}$. Then:\\[2mm]
\begin{tabular}{ccc}
\begin{tabular}{|c|l|}
\hline
$G$ & $\mathcal{H}_{\Q}(2,G)$ \\
\hline\hline
\multirow{6}{*}{$\cC_1$} & $\cC_3$  \\
\cline{2-2}
& $\cC_5$    \\
\cline{2-2}
& $\cC_7$   \\
\cline{2-2}
& $\cC_9$   \\
\cline{2-2}
& $\cC_3,\cC_3$  \\
\cline{2-2}
& $\cC_3,\cC_5$   \\
\hline
\hline
\multirow{14}{*}{$\cC_2$} & $\cC_2\times\cC_2$ \\
\cline{2-2}
& $ \cC_2\times\cC_6 $   \\
\cline{2-2}
& $ \cC_2\times\cC_{10} $  \\
\cline{2-2}
& $ \cC_2\times\cC_2, \cC_6 $ \\
\cline{2-2}
& $\cC_2\times\cC_2 , \cC_{10}$\\
\cline{2-2}
& $ \cC_2\times\cC_6, \cC_6 $  \\
\cline{2-2}
& $ \cC_2\times\cC_2, \cC_4, \cC_4 $ \\
\cline{2-2}
& $ \cC_2\times\cC_2, \cC_6, \cC_6 $   \\
\cline{2-2}
& $ \cC_2\times\cC_2, \cC_8, \cC_8 $   \\
\cline{2-2}
& $ \cC_2\times\cC_2, \cC_4, \cC_8 $   \\
\cline{2-2}
& $\cC_2\times\cC_2, \cC_4, \cC_{12} $ \\
\cline{2-2}
& $ \cC_2\times\cC_2, \cC_4,\cC_{16} $ \\
\cline{2-2}
& $ \cC_2\times\cC_6, \cC_4, \cC_4 $   \\
\cline{2-2}
& $ \cC_2\times\cC_2, \cC_4, \cC_4, \cC_6$  \\
\cline{2-2}
\hline
\end{tabular}
&
\begin{tabular}{|c|l|}
\hline
$G$ & $\mathcal{H}_{\Q}(2,G)$ \\
\hline\hline
\multirow{2}{*}{$\cC_3$}  & $ \cC_{15}$  \\
\cline{2-2}
& $  \cC_3\times\cC_3 $  \\
\hline\hline
\multirow{7}{*}{$\cC_4$}  & $ \cC_2\times\cC_4$   \\
\cline{2-2}
& $ \cC_2\times\cC_8 $   \\
\cline{2-2}
& $ \cC_2\times\cC_{12} $   \\
\cline{2-2}
& $ \cC_4\times\cC_4 $   \\
\cline{2-2}
& $ \cC_2\times\cC_4,\cC_{12} $   \\
\cline{2-2}
& $ \cC_2\times\cC_4, \cC_8, \cC_8 $    \\
\cline{2-2}
&   $ \cC_2\times\cC_8, \cC_8, \cC_8 $  \\
\hline\hline
$\cC_5$ & $ \cC_{15}$  \\
\hline\hline
\multirow{3}{*}{$\cC_6$}  & $ \cC_2\times\cC_6 $   \\
\cline{2-2}
& $ \cC_2\times\cC_6, \cC_3\times\cC_6 $   \\
\cline{2-2}
& $\cC_2\times\cC_6, \cC_{12}, \cC_{12}$   \\
\cline{2-2}
\hline\hline
\multirow{2}{*}{$\cC_8$}  & $ \cC_2\times\cC_8 $  \\
\cline{2-2}
& $\cC_2\times\cC_8, \cC_{16}, \cC_{16}$   \\
\hline\hline
$\cC_{10}$ & $ \cC_2\times\cC_{10}$  \\
\hline\hline
$\cC_{12}$ & $ \cC_2\times\cC_{12}$  \\
\hline
\end{tabular}
&
\begin{tabular}{|c|l|}
\hline
$G$ & $\mathcal{H}_{\Q}(2,G)$ \\
\hline\hline
\multirow{9}{*}{$\cC_2\times\cC_2$}  & $ \cC_2\times\cC_4 $   \\
\cline{2-2}
& $ \cC_2\times\cC_6 $   \\
\cline{2-2}
& $ \cC_2\times\cC_8$    \\
\cline{2-2}
& $ \cC_2\times\cC_{12} $    \\
\cline{2-2}
& $ \cC_2\times\cC_4, \cC_2\times\cC_4 $   \\
\cline{2-2}
& $ \cC_2\times\cC_4, \cC_2\times\cC_6 $    \\
\cline{2-2}
& $ \cC_2\times\cC_4, \cC_2\times\cC_8 $    \\
\cline{2-2}
& $ \cC_2\times\cC_4, \cC_2\times\cC_4, \cC_2\times\cC_4 $   \\
\cline{2-2}
& $ \cC_2\times\cC_4, \cC_2\times\cC_4, \cC_2\times\cC_8 $    \\
\hline\hline
\multirow{5}{*}{$\cC_2\times\cC_4$} & $ \cC_2\times\cC_8 $   \\
\cline{2-2}
& $ \cC_4\times\cC_4 $   \\
\cline{2-2}
& $ \cC_2\times\cC_8, \cC_4\times\cC_4 $   \\
\cline{2-2}
& $ \cC_2\times\cC_8, \cC_2\times\cC_8$    \\
\cline{2-2}
& $ \cC_2\times\cC_8, \cC_2\times\cC_8, \cC_4\times\cC_4 $ \\
\hline\hline
$\cC_2\times\cC_6$ & $ \cC_2\times\cC_{12}$  \\
\hline\end{tabular} 
\end{tabular}\\[2mm]
\end{theorem}

In particular, we obtain the following corollary:

\begin{corollary}
If $E$ is an elliptic curve defined over $\Q$, then there are at most four quadratic fields $K_i$, $i=1,\dots,4$, such that $E(K_i)_{\tors}\ne E(\Q)_{\tors}$. That is,
$$
\max_{G \in \Phi(1)} \left\{ \#S \,|\, S \in \mathcal{H}_{\Q}(2,G) \right\} = 4.
$$
\end{corollary}

We would like to mention this last result has also been proved independently by Najman \cite{Najman2014preprint}. His proof uses a very different kind of argument and, in particular, Theorems \ref{teo2} and \ref{teo3} do not follow from his results.

\

\noindent {\bf Acknowledgements.}  Both authors are grateful to Noam Elkies, for his insight in the problem concerning curves with $\cC_2 \times \cC_6$ torsion, and in particular for pointing out to them the parametrization in \cite{Elkies}. Also, Yasutsugu Fujita was very kind to explain to us in detail his argument for Proposition \ref{ClEx} and we thank him for this here. Last, the referees this paper was sent to did a painstaking and exhaustive work which greatly improved its overall quality, and both authors are enormously grateful for that.

\section{Some technical results}

Aside from the above main results, a number of auxiliary results are needed for our arguments. 

We already mentioned this result by Fujita:

\begin{theorem}\label{FJNT1}
\cite[Theorem 2]{Fujita2005} Let $E$ be an elliptic curve over $\Q$. Then, the torsion subgroup $E({\mathbb Q}(2^\infty))_{\tors}$ is isomorphic to one of the following $20$ groups:
$$
\begin{array}{ll}
\cC_N & \mbox{ for } N=1,3,5,7,9,15; \\
\cC_2 \times \cC_{2N} & \mbox{ for } N=1,...,6,8; \\
\cC_4 \times \cC_{4N} & \mbox{ for } N=1,...,4; \\
\cC_{2N} \times \cC_{2N} & \mbox{ for } N=3,4; \\
\cC_3 \times \cC_3\,.
\end{array}
$$
\end{theorem}

In the same paper one can find the following useful result:

\begin{proposition}\label{FJNT2}
\cite[Proposition 11]{Fujita2005} Let $E$ be an elliptic curve over $\Q$ such that $E(\Q)_{\tors}$ is cyclic. Then $\cC_8 \times \cC_8 \nleq E({\mathbb Q}(2^\infty))_{\tors}$.
\end{proposition}

A classical result which could be found, for instance, in \cite[Corollary 8.1.1]{Silverman2009} is the following:

\begin{proposition}\label{nthroot}
Let $E$ be an elliptic curve over a number field $L$. If $\cC_m \times \cC_m = E[m] \le E(L)$, then $L$ contains the cyclotomic field generated by the $m$--th roots of unity.
\end{proposition}

In another paper by Fujita \cite{Fujita2004}, the following two results can be found:

\begin{theorem}\label{FAA1}
\cite[Theorem 1]{Fujita2004} Let $E$ be an elliptic curve over $\Q$ such that $E(\Q)_{\tors}$ is non-cyclic.\\
\indent $\bullet$ If $E(\Q)_{\tors} = \cC_2 \times \cC_8$, then $E({\mathbb Q}(2^\infty))_{\tors} = \cC_4 \times \cC_{16}$.\\
\indent $\bullet$ If $E(\Q)_{\tors} = \cC_2 \times \cC_6$, then $E({\mathbb Q}(2^\infty))_{\tors} = \cC_4 \times \cC_{12}$.\\
\indent $\bullet$ If $E(\Q)_{\tors} = \cC_2 \times \cC_4$, then $E({\mathbb Q}(2^\infty))_{\tors} \in \left\{ \cC_4 \times \cC_8, \, \cC_8 \times \cC_8 \right\}$.\\
\indent $\bullet$ If $E(\Q)_{\tors} = \cC_2 \times \cC_2$, then $E({\mathbb Q}(2^\infty))_{\tors} \in \left\{ \cC_4 \times \cC_4, \, \cC_4 \times \cC_8, \, \cC_8 \times \cC_8, \, \cC_4 \times \cC_{12}, \, \cC_4 \times \cC_{16} \right\}$.\\
\end{theorem}

\begin{proposition}\label{ClEx}
\cite[Final Remark]{Fujita2004} The minimal $d$ for which the following groups can be realized as $E(L_{d} )_{\tors}$ with some elliptic curve $E$ defined over $\Q$, having non--cyclic rational torsion, and some polyquadratic field $L_{d}$ with $[L_d: \Q] = 2^d$, is:
\begin{enumerate}
\item $d= 4$ for $\cC_4 \times \cC_{16}$.
\item $d = 3$ for $\cC_4 \times \cC_{12}$.
\item $d = 4$ for $\cC_8 \times \cC_8$.
\item For all other types, we have $d_m=2$.
\end{enumerate}
\end{proposition}

\section{On $2$--divisibility}

In this section we are going to use two methods that allow us to decide when there exists a point (or where to look for it) which divides by two a given point of some order. The first method is classical in the literature of elliptic curves \cite[Theorem 4.2]{Knapp1992}. It allows us to decide if a point defined over a number field $L$ containing $\Q(E[2])$ is half a point over $L$ too.

\begin{lemma}\label{2dmknapp}
Let $E$ be an elliptic curve defined over a number field $L$ given by
$$
E:\, Y^2=(X-\alpha)(X-\beta)(X-\gamma),
$$
with $\alpha,\beta,\gamma \in L$. For $P=(x_0,y_0)\in E(L)$, there exists $Q\in E(L)$ such that $2Q=P$ if and only if $x_0-\alpha,x_0-\beta$ and $x_0-\gamma$ are all squares in $L$.
\end{lemma} 

For our concerns, this will apply specifically to the following situation:

\begin{proposition}\label{2dmnc}
Assume we have an elliptic curve 
$$
E: \ Y^2=X(X-A)(X-B), \quad A,B  \in \Q
$$
and $\cC_2\times\cC_2 \leq E(\Q)_{\tors}$ and there are no points of order $4$ in $E(\Q)$. Then, there are $1$, $2$ or $3$ quadratic fields $K$ with $\cC_2\times\cC_4\leq E(K)_{\tors}$. All three cases can appear.
\end{proposition}

\begin{proof}
Assume that the elliptic curve has $\cC_2\times\cC_4\leq E(K)_{\tors}$, with $K = \Q(\sqrt{D})$. Let us first assume that the point who gets divided by two is $(0,0)$. That is, there is a certain $Q\in E(K)$ such that $2Q=(0,0)$. By the previous lemma $0, -A, -B$ are then squares in $K$. This amounts to the existence of $a,b\in\Q$ such that one of the mutually exclusive pairs of equalities holds:
$$
\{ -A=a^2D,\, -B=b^2\}\,\,\mbox{or} \,\, \{ -A=a^2,\, -B=b^2D\}\,\,\mbox{or}\,\, \{-A=a^2D,\, -B=b^2D\}. 
$$

Of these cases, there is only one possible squarefree $D$ satisfying the conditions. The same goes if the divided point is $(A,0)$ (change $\{A,B\}$ for $\{A,A-B\}$) and if it is $(B,0)$. All in all there can be $1$, $2$ or $3$ quadratic extensions where the torsion contains $\cC_2\times\cC_4$.

In Table $1$ (see the appendix for an explanation of the table) one can find an example for each of the three circumstances.
\end{proof}

The second technique is taken from Jeon et al. \cite{Jeon-Kim-Lee2013}. This method allows to find, given a point defined over a number field $F$, an extension $L/F$ and a point defined over $L$ such that it is half of the given point. 
\begin{proposition}\label{newprop}
Let $E$ be an elliptic curve defined over a number field $F$ given by the Weierstrass equation:
$$
E: \ Y^2= X^3+AX^2+BX+y_0^2,
$$
and $P=(0,y_0)\in E(F)$. Let $\alpha$ be a root of the quartic polynomial  
$$
q(x)=x^4-2Ax^2-8y_0x+A^2-4B.
$$
Then the point $Q=((\alpha^2-A)/2,\alpha (\alpha^2-A)/2-y_0)\in E(L)$, where $L=F(\alpha)$, and $2Q=P$.
\end{proposition}
 
It is not difficult to check that the elliptic curve $E$ and the one defined by the quartic polynomial $q(x)$, $v^2=q(u)$, are isomorphic over $F$. Then, thanks to  \cite[Appendix A.2]{Gonzalez-Jimenez2014}, we know that $q(x)$ splits over a quadratic extension of $F$ for each $2$--torsion point of $E$ defined over $F$.

We will apply this procedure to points of even order $N$. Note that if $E(\Q)_{\tors}$ is cyclic and $P,P'$ are two generators of this cyclic group, then if there exist a number field $L$ and a point $Q \in E(L)$ with $2Q=P$, then there must also be some $Q' \in E(L)$ with $2Q'=P'$. That is, the $2$--divisibility holds for either all generators or for none of them.

\subsection{The case $N=2$} 

\begin{lemma}\label{2dm2}
Let 
$$
E:\ Y^2=X(X^2+AX+B)
$$ 
be an elliptic curve defined over $\Q$ with $E(\Q)_{\tors}= \cC_2$. Then, there exists a quadratic field $K$ with $\cC_4\leq E(K)_{\tors}$ if and only if $B=s^2$ for some $s\in\Q$. 

Moreover, $K=K_{\pm} := \Q (\sqrt{A\pm 2 s})$ in this situation and $K_+ \neq K_-$.
\end{lemma} 

\begin{proof}
Using Proposition \ref{newprop}, with the point $(0,0)$, we get the roots of the corresponding quartic polynomial $q(x)$ which are
$$
\pm \sqrt{A\pm 2\sqrt{B}}.
$$
A necessary and sufficient condition then for a point $Q$ to exist over a quadratic field, with $2Q=(0,0)$, is $B=s^2$ for a certain $s\in\Q$. Should this be the case, $Q \in E(K)[4]$, with $K = \Q(\sqrt{A\pm 2 s})$.

Please note that we have implicitly assumed that there are no points of order $2$ in $E(K')$ other than $(0,0)$ that could be divided by $2$ over any quadratic field $K'$. In fact, this must always be the case, as from \cite[Thm. 5 (ii)]{Gonzalez-Jimenez-Tornero2014}, $G = \cC_2$ implies $\cC_2 \times \cC_4 \not\leq E(K')_{\text{tors}}$ for any quadratic field $K'$.

Let us check $K_+\neq K_-$ for all $s$. Assume $K_+ = K_-$. Then, $A^2-4s^2$ is a rational square. Therefore, $X^2+AX+s^2$ has two different rational roots. That is, $C_2\times C_2 \leq E(\Q)$, which is a contradiction.
\end{proof}

\subsection{The cases $N=4,6,8$} 

$\;$\\

Let $N\ge 4$ be an integer. We are given a curve $E$ defined over a number field $L$ (for our purposes it will mostly be $\Q$, but the result is more general) and a point $P \in E(L)$ of order $N$, and then we take the Tate normal form of $E$:
$$
\mathcal T_{b,c}:\ Y^2+(1-c)XY-bY=X^3-bX^2,
$$
where $P = (0,0)$. Changing coordinates by means of 
$$
X\longmapsto X  \quad ,\quad Y \longmapsto Y+\frac{c-1}{2}x+\frac{b}{2}\,;
$$ 
we obtain a Weierstrass model:
$$
\mathcal T_{b,c}: \ Y^2=X^3+AX^2+BX+C,
$$
with
$$
A=\frac{(c-1)^2-4b}{4}, \quad B=\frac{b(c-1)}{2}, \quad C=\frac{b^2}{4}. 
$$
In particular $P=\left(0,-b/2\right)$. Then the quartic polynomial $q(x)$ which characterizes the existence of $Q$ such that $2Q=P$ (see Proposition \ref{newprop}) is now:
\begin{equation}\label{qq}
\begin{array}{rcl}
q(x) &=& x^4+\frac{1}{2}(-1 + 4b + 2c - c^2)x^2+4bx+ \\[1mm]
&& \qquad \quad + \frac{1}{16}(1 + 24b + 16b^2 - 4c - 16bc + 6c^2 - 8bc^2 - 4c^3 + c^4).
\end{array}
\end{equation}

The Tate normal form also has an important feature, as it parametrizes the different curves defined over the rationals with a common torsion structure \cite{Husemoller2004}. Precisely, if $\cC_N \leq E(\Q)$, there exists $t\in\Q$ such that $E$ is $\Q$-isomorphic to  $\mathcal T_{b,c}$ where:

\begin{itemize}
\item $c=0$ and $b=t$ if $N=4$; 
\item $c=t$ and $b=t^2+t$ if $N=6$;
\item $c=(2t-1)(t-1)/t$ and $b=(2t-1)(t-1)$ if $N=8$.
\end{itemize}

\begin{lemma}\label{2dm4}
Let $E$ be an elliptic curve defined over $\Q$ with $E(\Q)_{\tors}= \cC_4$. Let $t\in\Q$ such that $E$ is $\Q$--isomorphic to $\mathcal T_{t,0}$. Then, there exists a quadratic field $K$ with $E(K)_{\tors}= \cC_8$ if and only if $t=-s^2$ for some $s\in\Q$. 

Moreover, $K=K_{\pm} := \Q (\sqrt{1\pm 4 s})$ in this situation and $K_+ \neq K_-$.
\end{lemma} 

\begin{proof}
In this case, the roots of the quartic polynomial given at (\ref{qq}) are
$$
\sqrt{-t}\pm\frac{1}{2}\sqrt{1+4\sqrt{-t}} \quad,\quad -\sqrt{-t}\pm\frac{1}{2}\sqrt{1-4\sqrt{-t}}
$$

A necessary and sufficient condition then for a point $Q$ to exist over a quadratic field, with $2Q=(0,0)$, is $t=-s^2$ for a certain $s\in\Q$. Should this be the case, $Q \in E(K_\pm)[8]$, with $K_\pm = \Q(\sqrt{1\pm 4 s})$.

As above, it must be $(0,0)$ the point in $E[4]$ who gets divided by $2$. If there were a non--rational point $P \in E(K')$ of order $4$ over some quadratic field $K'$ such that there exists $Q \in E(K')$ with $2Q=P$, then $E(K')_{\tors}$ must be a group with an element $Q$ of order $8$ which does not generate the whole group (it does not generate $(0,0)$ in particular), which contradicts our assumption $E(K')_{\tors}= \cC_8$.

If $K_+ = K_-$, then $(1+4s)(1-4s)$ is a rational square. Therefore, $\Delta_E$ is a rational square. That is, $C_2\times C_2 \leq E(\Q)$, which is a contradiction. 
\end{proof}

\noindent {\bf Remark.--}  Note that the assumption $E(K)_{\tors}= \cC_8$ is indeed necessary. Since if we relax  this hypothesis to $E(K)_{\tors}\le \cC_8$, Lemma \ref{2dm4} is false:  the elliptic curve \texttt{240d6} has torsion subgroup $\cC_4$ (resp. $\cC_2\times\cC_8$, $\cC_8$, $\cC_8$) over $\Q$ (resp. $\Q(\sqrt{-1})$, $\Q(\sqrt{6})$, $\Q(\sqrt{-6})$) (see Table \ref{tablagrande}).

\begin{lemma}\label{2dm6}
Let $E$ be an elliptic curve defined over $\Q$ with $E(\Q)_{\tors} = \cC_6$. Let $t\in\Q$ such that $E$ is $\Q$--isomorphic to $\mathcal T_{t^2+t,t}$. Then, there exists a quadratic field $K$ with $ \cC_{12} \leq E(K)_{\tors}$ if and only if $t=-s^2$ for some $s\in\Q$. 

Moreover, $K=K_{\pm} := \Q(\sqrt{(1\pm s)(1\mp 3s)})$ in this situation and $K_+ \neq K_-$.
\end{lemma} 

\begin{proof}
In this case, the roots of the polynomial given at (\ref{qq}) are
$$
\sqrt{-t}\pm\frac{1}{2}\sqrt{(1+t)(1-4\sqrt{-t}-3t)} \quad,\quad-\sqrt{-t}\pm\frac{1}{2}\sqrt{(1+t)(1+4\sqrt{-t}-3t)} 
$$

A necessary and sufficient condition then for a point $Q$ to exist over a quadratic field, with $2Q=P$, is $t=-s^2$ for a certain $s\in\Q$. Should this be the case: $Q \in E(K_\pm)[12]$, with $K_{\pm} = \Q(\sqrt{(1\pm s)(1\mp 3s)})$. 

Again, the point in $E[6]$ who gets divided by $2$ must be rational. This time it is easier, as the only group in $\Phi_{\Q}(2,\cC_6)$ with elements of order $12$ is precisely $\cC_{12}$, so the only two available points are $(0,0)$ and its inverse, which yield the same situation.

If $K_+ = K_-$ for some $s$, there exists $r\in\Q$ with 
$$
(1+s)(1-3s)=r^2(1-s)(1+3s).
$$ 
That is to say, the equation
$$
C:\ z^2=(1-s^2)(1-9s^2)
$$
has a non--trivial rational solution, $s\ne 0,\pm 1,\pm 1/3$ (these solutions correspond to Tate models which do not yield elliptic curves). $C$ defines then an elliptic curve with at least $8$ rational points: $6$ trivial ones, and $2$ more at infinity. But $C$ is $\Q$--isomorphic to \texttt{24a1}, whose Mordell group is $\cC_2\times \cC_4$. Therefore, the affine points in $C(\Q)$ correspond to the trivial points.
\end{proof}

\begin{lemma}\label{2dm8}
Let $E$ be an elliptic curve defined over $\Q$ with $E(\Q)_{\tors} = \cC_8$. Let $t\in\Q$ such that $E$ is $\Q$--isomorphic to $\mathcal T_{(2t-1)(t-1),(2t-1)(t-1)/t}$. Then, there exists a quadratic field $K$ with $ \cC_{16} \leq E(K)_{\tors}$ if and only if $t=s^2/(s^2+1)$ for some $s\in\Q$. 

Moreover, $K=K_{\pm} := \Q(\sqrt{(s^4-1)(-1\pm 2s+s^2)})$ in this situation and $K_+ \neq K_-$.
\end{lemma} 

\begin{proof}
In this case, the roots of the polynomial given at (\ref{qq}) are
$$
\begin{array}{r}
\sqrt{t(1-t)}\pm\frac{1}{2t}\sqrt{(1-2t)(1-6t+4t^2-4t\sqrt{t(1-t)})},\\[1mm]
 -\sqrt{t(1-t)}\pm\frac{1}{2t}\sqrt{(1-2t)(1-6t+4t^2-4t\sqrt{t(1-t)})}.
 \end{array}
$$
A necessary and sufficient condition then for a point $Q$ to exist over a quadratic field, with $2Q=P$, is $t(1-t)=s^2$ for a certain $s\in\Q$. This equation is a genus zero curve again, parametrized by:
$$
t=\frac{r^2}{r^2+1}\qquad,\qquad s=\frac{r}{r^2+1},
$$
for some $r\in\Q$. Should this be the case, $Q \in E(K_\pm)[12]$, with 
$$
K_{\pm} = \Q(\sqrt{(r^4-1)(-1\pm 2r+r^2)}).
$$

Once more, the point in $E[8]$ who gets divided by $2$ must be rational, as the only group in $\Phi_{\Q}(2,\cC_8)$ with elements of order $16$ is $\cC_{16}$.

Finally, let us check $K_+\ne K_-$ for all $s$. If not, there is some $r\in\Q$ with 
$$
(s^4-1)(-1+ 2s+s^2)=r^2(s^4-1)(-1- 2s+s^2)
$$
for a certain $s$. That implies the equation 
$$
C:\ z^2=(-1+ 2s+s^2)(-1- 2s+s^2)
$$ 
has a non--trivial rational solution (non--trivial meaning $s\ne 0$), as the trivial solutions match the Tate models which do not yield elliptic curves. $C$ defines an elliptic curve with at least $4$ rational points ($2$ trivial, $2$ at infinity), but in fact it is isomorphic to the curve \texttt{32a2} whose Mordell group is $\cC_2\times \cC_2$. Hence the affine points in $C(\Q)$ are just the trivial points and we are done.
\end{proof}

\section{Proof of theorem \ref{teo2}}

For a given $G \in \Phi(1)$ and $H\in \Phi_\Q \left(2,G \right)$, we calculate the number  $h$ of possible quadratic fields $K$ such that, for a given rational elliptic curve $E$ with $E(\Q)_{\tors}= G$, we have $E(K)_{\tors}= H$.

$\;$\\

\subsection{The cyclic case}

$\;$\\

$\bullet$ Clearly, if $H=\cC_{2}\times\cC_{2m}$ for some integer $m$, this can only happen over the quadratic field $K=\Q(\sqrt{\Delta_E})$. Note that $K$ is actually always a quadratic extension, as $\Q(E[2]) \neq \Q$. This rules out the cases:

\begin{itemize}
\item[$\circ$]  $G= \cC_2$, $H=\cC_2 \times \cC_{2m}$, with $m=1,3,5$; 
\item[$\circ$]  $G= \cC_4$, $H=\cC_2\times \cC_{4m}$, with $m=1,2,3$;
\item[$\circ$] $G= \cC_r$, $H=\cC_2 \times \cC_r$, with $r=6,8,10,12$.
\end{itemize}

$\bullet$ Assume $G=\cC_{2}$ and $H \leq \cC_{4n}$. Lemma \ref{2dm2} shows that there can be $1$ or $2$ quadratic fields in which this situation holds. When $H=\cC_4,\cC_8$ in fact both things can happen (see examples in Table \ref{tablagrande} at the appendix). 

However, for the remaining cases, the situation can only hold in one quadratic field.  Let us do with a little detail the case $H=\cC_{12}$, as the case $H=\cC_{16}$ is analogous. So we are assuming $G=\cC_2$ and $H=\cC_{12}$ for two different quadratic fields. Then, as we also have a quadratic field where the full $2$--torsion appears, $\cC_6 \times \cC_{12}$ should be a subgroup of one of the groups in $\Phi_\Q(2^\infty)$, and that is not possible from Theorem \ref{FJNT1}.

$\bullet$ If $G = \cC_{2n}$ and $H=\cC_{4n}$ for $n=2,3,4$, Lemmas \ref{2dm4},\ref{2dm6},\ref{2dm8} (respectively) show that there are exactly two quadratic fields where the appropriate torsion extension occurs.

$\bullet$ If $H=\cC_{4}\times\cC_{4}$ (resp. $H=\cC_{3}\times\cC_{3n}$, $n=1,2$) the quadratic field must be $K=\Q(\sqrt{-1})$ (resp. $K=\Q(\sqrt{-3})$) by 
\ref{nthroot}. This proves the cases 
\begin{enumerate}
\item[$\circ$] $G= \cC_4$, $H=\cC_4 \times \cC_4$; 
\item[$\circ$] $G= \cC_3$, $H=\cC_3 \times \cC_3$;
\item[$\circ$] $G= \cC_6$, $H=\cC_3 \times \cC_6$.
\end{enumerate}

$\bullet$ For any given $G = \cC_n$, $H=G \times \cC_{m}$ with $\gcd(n,m)=1$ can appear at most twice, since $E[m]=  \cC_{m}\times  \cC_{m}$. More precisely, if $m=5,7,9$ then only one quadratic field may extend the torsion in this way since, if there were two such quadratic fields, the cyclotomic field generated by the $m$--th roots of unity, $\Q(\zeta_m)$, should be a subfield of the corresponding biquadratic case from Proposition \ref{nthroot}, and that is not possible. This proves the cases:
\begin{enumerate}
\item[$\circ$] $G= \cC_1$, $H=\cC_m$, with $m=5,7,9$; 
\item[$\circ$] $G= \cC_2$, $H=\cC_{10}$.
\item[$\circ$] $G= \cC_3$, $H=\cC_{15}$.
\end{enumerate}
Now if $m=3$ then $H$ may appear once or twice. It actually happens twice in the following cases (see examples in Table \ref{tablagrande} at the appendix):
\begin{enumerate}
\item[$\circ$] $G= \cC_1$, $H=\cC_3$. 
\item[$\circ$] $G= \cC_2$, $H=\cC_6$.
\end{enumerate}

$\bullet$ There are only two cases remaining: $G = \cC_n$, $H = \cC_{3n}$ for $n=4,5$. Only one quadratic field is possible in these instances. If there were two quadratic fields where $H$ appears, then $\cC_n\times\cC_3\times\cC_3$ should be a subgroup of one of the groups  in $\Phi_\Q(2^\infty)$ for $n=4,5$; and that is impossible from Theorem \ref{FJNT1}.\\

\subsection{The non--cyclic case}

$\;$\\

Let $E$ be an elliptic curve defined over $\Q$ such that $E(\Q)_{\tors}=G$ where $G$ is the following:

$\bullet$ $G=\cC_{2}\times \cC_2$. If $H=\cC_{2}\times \cC_4$ there might be $1$, $2$ or $3$ quadratic extensions, following Proposition \ref{2dmnc} in the previous section. 

If $H=\cC_{2}\times \cC_{2n}$ with $n=3,6$ appears in two different quadratic extensions, then there are two independent points of order $3$ in $\Q(2^\infty)$. As a result, $\cC_6 \times \cC_6 \leq E(\Q(2^\infty))_{\tors}$, which contradicts Theorem \ref{FAA1}. 

If $H=\cC_{2}\times \cC_{8}$ for two different quadratic extensions, we must have two different points of order $8$. Let us call $L$ the composition field of these two quadratic extensions. There are two groups in $\Phi_\Q (2^\infty)$ with more than one element of order $8$: $\cC_{4}\times \cC_{8}$ and $\cC_{8}\times \cC_{8}$. But the first one is not our case: looking at the lattice of subgroups of $\cC_4 \times \cC_8$ one can realize that both $\cC_2 \times \cC_8$ have a common subgroup $\cC_2 \times \cC_4$, while the intersection (in our case) should only be $G = \cC_2 \times \cC_2$. This implies $E(L)_{\tors}$ had to be $\cC_{8}\times \cC_{8}$ and Proposition \ref{ClEx} tells us that under these circumstances $[L:\Q] \geq 16$. Hence only one quadratic extension with $H=\cC_{2}\times \cC_{8}$ can occur.

$\bullet$ $G=\cC_{2}\times \cC_4$. As we mentioned above, if $H=\cC_{4}\times \cC_4$ the only possible extension is $\Q(\sqrt{-1})/\Q$. 

When $H=\cC_{2}\times \cC_{8}$ the first part of Lemma \ref{2dm4} can be applied verbatim and it shows that $1$ or $2$ extensions can appear (both things occur). 

$\bullet$ $G=\cC_{2}\times \cC_6$. The only group extension, by Theorem \ref{teo1} is $H=\cC_{2}\times \cC_{12}$. Lemma \ref{2dm6} tells us (the first part) that either one or two relevant quadratic extensions may appear. 

Also, from Theorem \ref{FAA1} we know that $E(\Q(2^\infty))_{\tors} = \cC_4 \times \cC_{12}$, and by Proposition \ref{ClEx} that $E(L)_{\tors} = \cC_4 \times \cC_{12}$ implies $[L:\Q] \geq 8$. 

But, if there were two quadratic extensions, $K_1,K_2$ with $E(K_i)_{\tors} = \cC_2 \times \cC_{12}$, let us write $F$ the composite of $K_1$ and $K_2$ (in particular, $[F:\Q]=4$). Then clearly $E(F)_{\tors} = \cC_4 \times \cC_{12}$, because it must be contained in $E(\Q(2^\infty))_{\tors}$ and it should be strictly bigger than both $E(K_i)_{\tors}$. 

This is a contradiction and therefore, only one quadratic extension $K$ can appear with $E(K)_{\tors}= H =\cC_{2}\times \cC_{12}$.

\

\noindent {\bf Remark.--} These two last cases can also be found in \cite{Kwon1997}, but the proofs there are longer, as we can take advantage of the many results which have appeared concerning this matter since (specially those in \cite{Fujita2004, Fujita2005}).

\section{Proof of theorem \ref{teo3}}

Now we are going to prove Theorem \ref{teo3}. For this purpose, for a given $G\in\Phi(1)$ let us build a set $\mathcal S(G)$ consisting of the groups $H\in \Phi_\Q(2,G) \setminus \{G\}$, repeated as many times as the number of possible quadratic fields where $H$ appears in Theorem \ref{teo2}. Our task is checking, for any subset $S\in \mathcal S(G)$ if $S$ belongs to $\mathcal H_\Q(2,G)$ or not. 

\

\noindent {\bf Example.--} As 
$$
\Phi_\Q(2,\cC_1) = \left\{ \; \cC_1, \; \cC_3, \; \cC_5, \; \cC_7, \; \cC_9 \;  \right\}
$$ 
and Theorem \ref{teo2} tells us that two quadratic extensions can appear with torsion group $\cC_3$, we have
\begin{eqnarray*}
\mathcal{S} \left( \cC_1 \right) 
&=& \Big\{ \left[ \cC_3 \right]; \;  \left[ \cC_5 \right]; \;  \left[ \cC_7 \right]; \;  \left[ \cC_9 \right]; \; \left[ \cC_3, \cC_3 \right]; \; \left[ \cC_3, \cC_5 \right]; \; \left[ \cC_3, \cC_7 \right]; \; \left[ \cC_3, \cC_9 \right]; \\
&& \quad \left[ \cC_5, \cC_7 \right]; \; \left[ \cC_5, \cC_9 \right]; \; \left[ \cC_7, \cC_9 \right]; \; \left[ \cC_3, \cC_3, \cC_5 \right]; \; \left[ \cC_3, \cC_3, \cC_7 \right]; \; \left[ \cC_3, \cC_3, \cC_9 \right]; \\
&& \quad \left[ \cC_3, \cC_5, \cC_7 \right]; \; \left[ \cC_3, \cC_5, \cC_9 \right]; \; \left[ \cC_3, \cC_7, \cC_9 \right]; \;
\left[ \cC_5, \cC_7, \cC_9 \right]; \; \left[ \cC_3, \cC_3, \cC_5, \cC_7 \right]; \\ 
&& \quad \left[ \cC_3, \cC_3, \cC_5, \cC_9 \right]; \; \left[ \cC_3, \cC_3, \cC_7, \cC_9 \right]; \; \left[ \cC_3, \cC_5, \cC_7, \cC_9 \right]; \; \left[ \cC_3, \cC_3, \cC_5, \cC_7, \cC_9 \right] \; \Big\}.
\end{eqnarray*}

Mind that at Table \ref{tablagrande} we have (for all $G \in \Phi(1)$) examples of elliptic curves over $\Q$ satisfying the conditions in Theorem \ref{teo3}, for any $S\in\mathcal H_\Q(2,G)$. Therefore, now we have to prove that there does not exist any other possible $S\in\mathcal S(G)$. 

\vspace{3mm}

\noindent {\bf Remark.--} Let be $G\in\Phi(1)$ cyclic and of even order. Then, for any $S\in\mathcal H_\Q(2,G)$  there always exists a unique non-cyclic $H\in S$, the one corresponding to $\Q(E[2])$ (a quadratic extension in this case), where $E$ is the elliptic curve associated to $S$.

\subsection{The groups $\cC_7,\cC_9,\cC_2\times\cC_8$}

$\;$\\

These are the easiest cases, since by Theorem \ref{teo1} we have that these groups are stable under all quadratic extensions. Therefore, in these cases,
$$
\mathcal H_\Q(2,G)= \emptyset.
$$ 

\subsection{The groups $\cC_5, \cC_{10}, \cC_{12}, \cC_2\times\cC_{6}$}

$\;$\\

Using Theorem \ref{teo2}, these cases are almost as easy as the previous ones, since we have that $\mathcal S(G)$ has only one element and we have examples in Table \ref{tablagrande} for any of those cases, we obtain that 
$$
\mathcal H_\Q(2,G)=\mathcal S(G).
$$

\subsection{The group $\cC_1$}

$\;$\\

Consider the groups in $\Phi_\Q(2,\cC_1)$. Mind that the intersection of two groups must be trivial in this case, hence we must look for (two or more) elements in $\Phi_\Q (2,\cC_1)$, other than $\cC_1$, such that their product lies in $\Phi_{\mathbb Q}(2^{\infty})$. From that, we easily deduce that 
$$
\mathcal H_\Q(2,\cC_1) = \Big\{ \, [\cC_3]; \; [\cC_5]; \; [\cC_7]; \; [\cC_9]; \; [\cC_3, \cC_3]; \; [\cC_3, \cC_5] \Big\}.
$$

\subsection{The group $\cC_3$}

$\;$\\

From all cases in $\mathcal{S}(\cC_3)$, the only case to discard is $S=[ \cC_3 \times \cC_3, \; \cC_{15}]$. In that case, $\cC_3 \times \cC_{15}$ should be a subgroup of some group in $\Phi_{\mathbb Q}(2^{\infty})$. But this does not happen.
$$
\mathcal H_\Q(2,\cC_3)  = \Big\{ \, [\cC_3 \times \cC_3]; \; [\cC_{15}] \, \Big\}.
$$

\subsection{The group $\cC_8$}

$\;$\\

By the previous remark, Theorem \ref{teo2} and Lemma \ref{2dm8} we have that the only possible subsets in $\mathcal S(\cC_8)$ are $[\cC_2 \times \cC_8]$ and $[\cC_2 \times \cC_8,\cC_{16},\cC_{16}]$. Mind that $\cC_{16}$ appears twice or it does not appear at all, from Lemma \ref{2dm8}. Since we have examples in Table \ref{tablagrande} for those cases, we have proved:
$$
\mathcal H_\Q(2,\cC_8)  = \Big\{ \, [\cC_2 \times \cC_8]; \; [\cC_2 \times \cC_8,\cC_{16},\cC_{16}]\, \Big\}.
$$

\subsection{The group $\cC_2 \times \cC_4$}

$\;$\\

As previously, we have examples in Table \ref{tablagrande} for any subset in $\mathcal S(\cC_2\times \cC_4)$, which proves:
\begin{eqnarray*}
\mathcal H_\Q(2,\cC_2\times \cC_4) &=& \Big\{ \, [\cC_2 \times \cC_8]; \, [\cC_4 \times \cC_4]; \, [\cC_2 \times \cC_8, \cC_2 \times \cC_8]; \, [\cC_2 \times \cC_8, \cC_4 \times \cC_4]; \\
&& \quad \quad [\cC_2 \times \cC_8, \cC_2 \times \cC_8, \cC_4 \times \cC_4] \, \Big\}.
\end{eqnarray*}

\subsection{The group $\cC_6$}

$\;$\\

From the examples in Table \ref{tablagrande} the only case to discard is $S=[\cC_2\times\cC_6,\cC_3\times\cC_6,\cC_{12}, \cC_{12}]$ (as above, Lemma \ref{2dm6} implies that $\cC_{12}$ appears twice if it does). But if there exists an elliptic curve $E$ over $\Q$ such that over four quadratic fields has those torsion subgroups, then $\cC_3 \times \cC_{12}$ is a subgroup of $E(\Q(2^\infty))_{\tors}$. But no group of $\Phi_\Q(2^\infty)$ has such subgroups from Theorem \ref{FJNT1}. Therefore we have proved:
$$
\mathcal H_\Q(2,\cC_6)  = \Big\{ \, [\cC_2 \times \cC_6]; \; [\cC_2 \times \cC_6,\cC_{3}\times\cC_{6}]; \; [\cC_2 \times \cC_6,\cC_{12},\cC_{12}] \, \Big\}.
$$

\subsection{The group $\cC_4$}

$\;$\\

There must always be exactly one non--cyclic group, and Lemma \ref{2dm4} tells us that $\cC_8$, if it appears in a quadratic extension, then it appears in two quadratic extensions. So, a quick comparison between $\mathcal{S}(\cC_4)$ and $\mathcal{H}_\Q(2,\cC_4)$ in Theorem \ref{teo3} tells us that it suffices to prove two assertions. 

First, there does not exist $S\in \mathcal{H}_\Q(2,\cC_4)$ such that one of the following facts happens:
\begin{itemize}
\item $H_1,H_2\in S$ such that $\cC_8\le H_1$ and $\cC_{12}\le H_2$;
\item $H_1,H_2\in S$ such that $H_1=H_2=\cC_{12}$;
\end{itemize}

Note that there does not exist $H\in \Phi_\Q(2^\infty)$ with elements of order $8$ and $12$. This proves the first point. On the other hand, $\cC_{12}$ cannot appear twice in an element in $S$, since that would imply there should exist $H\in \Phi_\Q(2^\infty)$ with $\cC_3\times\cC_{12} \leq H$. But that is impossible too from Theorem \ref{FJNT1}. 

Second and last, we need to prove that if $\cC_4 \times \cC_4\in S$, then $S=[\cC_4 \times \cC_4]$. That is, we have to discard the following elements in $\mathcal{S}(\cC_4)$: 
$$
\left[ \cC_4\times\cC_{4}, \; \cC_{12} \right], \quad \left[ \cC_4\times\cC_{4}, \; \cC_{8}, \; \cC_{8} \right]. 
$$

Let us prove first $\left[ \cC_4\times\cC_{4}, \; \cC_{12} \right] \notin \mathcal{S}(\cC_4)$. Suppose that there exists an elliptic curve $E$ over $\Q$ and a squarefree integer $D$ such that $E(\Q(\sqrt{D}))_{\tors} = \cC_{12}$ and $E(\Q(\sqrt{-1}))_{\tors} = \cC_4\times\cC_{4}$. Let us denote by $L = \Q (\sqrt{D},\sqrt{-1} )$. In our situation $\cC_6\leq E_D(\Q)_{\tors}$ from \cite[Cor. 4]{Gonzalez-Jimenez-Tornero2014} and $\cC_2\times\cC_{6}\leq E_D(\Q(\sqrt{-1}))_{\tors}$. Let $t\in \Q$ be the relevant parameter in the Tate model of $E_D$ (the one we recalled in subsection 3.2). That is, we can find a $\Q$-isomorphism such that a model for $E_D$ is:
$$
Y^2=(X - t) \left(X^2-\frac{1}{4}( 3 t^2 + 2 t-1) X-\frac{t}{4}(t^2+2t+1) \right).
$$
Now, since $\cC_2\times\cC_2<E_D(\Q(\sqrt{-1}))_{\tors}$, this means the discriminant of $E_D$ is a square in $\Q(\sqrt{-1})$ (and not in $\Q$), which implies $(1+t)(1+9t)=-r^2$ for some $r\in\Q$. Parametrizing this conic we obtain 
$$
t=-\frac{81m^2+1}{9(9m^2+1)}
$$
for some $m\in\Q$. Taking this back to the equation above we have the points of order $2$: $(A\pm B \sqrt{-1},0),(t,0)$ where
\begin{equation}\label{eq1}
A= -\frac{4 (1 + 36 m^2 + 243 m^4)}{27 (1 + 9 m^2)^3}  \quad\mbox{and}\quad B=-\frac{24 (m + 9 m^3)}{27 (1 + 9 m^2)^3}.
\end{equation}
Using 
$$
E (\Q (\sqrt{D} ))_{\tors} = \cC_{12}, \quad E(\Q(\sqrt{-1}))_{\tors} = \cC_4\times\cC_{4},
$$ 
we have $E(L)_{\tors} = \cC_4\times \cC_{12}$ from Theorem \ref{FJNT1}. Therefore 
$$
E_D(L)_{\tors} = \cC_4\times \cC_{12},
$$
since $E$ and $E_D$ are isomorphic over $\Q (\sqrt{D})$. Let us prove that this is impossible. Assume that all the points of order $2$ can be divided by two in $L$. In particular, there should exist $\gamma\in L$ such that $A\pm B \sqrt{-1} =\gamma^2$. If 
$$
\gamma=a_0+a_1 \sqrt{-1}+a_2\sqrt{D}+a_3\sqrt{-D},
$$
then it is a straightforward computation to check that a necessary condition is that $\gamma=a+b\sqrt{-1}$ or $\gamma=a\sqrt{D}+b\sqrt{-D}$ for some $a,b\in\Q$. Assuming that $\gamma$ is of one of the forms above, the equality $A\pm B \sqrt{-1}=\gamma^2$ holds if and only if $A=(a^2-b^2)r$ and $B=2abr$, where $r=1$ or $r=D$. Solving this equations on the variables $a$ and $b$ and using the definition of $A$ and $B$ from (\ref{eq1}) we obtain
$$
a=\pm \frac{2m}{1+9m^2}\sqrt{\frac{2}{3r}}\left(1+27m^2\pm \sqrt{(1+9m^2)(1+81m^2)}\right)^{-\frac{1}{2}}.
$$
Then a necessary condition for $a \in \Q$ is that $(1+9m^2)(1+81m^2)=s^2$ for some $s\in\Q$. This equation defines an elliptic curve (\texttt{48a1}) over $\Q$, whose Mordell group is $\cC_2\times\cC_2$. But apart form the points at infinity, these points correspond to $m=0$, and this value gives us a Tate model which does not yield an elliptic curve (it corresponds to $t=-1/9$). This proves $\left[ \cC_4\times\cC_{4}, \; \cC_{12} \right] \notin \mathcal{S}(\cC_4)$.

Finally then, let us prove $\left[ \cC_4\times\cC_{4}, \; \cC_8, \; \cC_8 \right] \notin \mathcal{S}(\cC_4)$. That is, we have to prove that, if an elliptic curve $E$ over $\Q$ has $E(\Q)_{\tors}= \cC_4$ then there does not exist a squarefree integer $D$ such that $E (\Q(\sqrt{D} ))_{\tors} = \cC_{8}$ and $E(\Q(\sqrt{-1}))_{\tors} = \cC_4\times\cC_{4}$. 

If $\cC_8 = E(K)_{\tors}$ for some quadratic field $K$ then $t=-s^2$ for some $s\in\Q$ from Lemma \ref{2dm4}; where $t$ is the relevant parameter in the Tate model of $E$. That is:
$$
E:\, Y^2=X^3+\frac{1}{4}\left( 1+4s^2 \right) X^2+\frac{s^2}{2}X+\frac{s^4}{4}.
$$
As $E(\Q(\sqrt{-1}))_{\tors} = \cC_4\times\cC_{4}$ it must have full $2$--torsion over $\Q(\sqrt{-1})$ and that means $\Delta_E$ is a square in $\Q(\sqrt{-1})$. This implies $1-16s^2$ is a square in $\Q(\sqrt{-1})$ (and not in $\Q$), and hence we can write  
$$
1-16s^2=-r^2, 
$$ 
for some $r \in \Q$. Parametrizing this conic we obtain
$$
s=\frac{m^2+4m+5}{4(m+1)(m+3)},\qquad r=\frac{2(2+m)}{(m+1)(m+3)},
$$
for some $m\in\Q$. Taking this back to the equation of $E$ we find that the full $2$--torsion is given by points $(\alpha_i,0)$, $i=1,2,3$, where
$$
\alpha_1=-\frac{(m+2+\sqrt{-1})^2}{8(m+1)(m+3)},\quad 
\alpha_2=-\frac{(m+2-\sqrt{-1})^2}{8(m+1)(m+3)},\quad 
\alpha_3=-\frac{(5+4m+m^2)^2}{16(m+1)^2(m+3)^2}.
$$
As $E(\Q(\sqrt{-1}))_{\tors} = \cC_4\times\cC_4$, all these points can be halved in $\Q(\sqrt{-1})$, so, by Lemma \ref{2dm2}, $\alpha_i-\alpha_j$ must be a square in $\Q(\sqrt{-1})$ for all $i,j\in\{1,2,3\}$. In particular
$$
\alpha_1-\alpha_2=-\frac{(m+2)}{2(m+1)(m+3)}\sqrt{-1}.
$$
That is $\alpha_1-\alpha_2=r\sqrt{-1}$ where $r\in\Q$. So, if $\alpha_1-\alpha_2=\beta^2$ for some $\beta=a+b\sqrt{-1}\in\Q(\sqrt{-1})$, it must be $b=\pm a$, and $\beta=a\pm a \sqrt{-1}$. Then
$$
-\frac{(m+2)}{2(m+1)(m+3)}=\pm 2a^2,
$$
otherwise said,
$$
(m+1)(m+2)(m+3)=\pm z^2,
$$
for some $z\in\Q$. These two equations define elliptic curves over $\Q$ and in fact both are isomorphic to \texttt{32a2}, whose Mordell group is $\cC_2\times\cC_2$. So, the only available solutions are the trivial ones ($z=0$) given by $m=-1,-2,-3$. But $m=-1,-3$ are not available in the parametrization above (as they divide the numerator of $s$), while $m=-2$ gives us a Tate model which does not yield an elliptic curve (it corresponds to $t=-1/16$). 

Therefore we have proved:
\begin{eqnarray*}
\mathcal H_\Q(2,\cC_4)  &=& \Big\{ \, [\cC_2\times\cC_4]; \; [\cC_2\times\cC_8]; \;  [\cC_2\times\cC_{12}], \; [\cC_4\times\cC_4]; \\
&& \quad [\cC_2\times\cC_4,\cC_{12}]; \; [\cC_2\times\cC_4, \cC_8, \cC_8]; \; [\cC_2\times\cC_8, \cC_8, \cC_8] \Big\}.
\end{eqnarray*}

\subsection{The group $\cC_2 \times \cC_2$}

\

\vspace{1mm}
As before, a comparison between $\mathcal{S}(\cC_2 \times \cC_2)$ and $\mathcal{H}_\Q (2,\cC_2 \times \cC_2)$ (shown in Table \ref{tablagrande} at the appendix) tells us that the proof for this case amounts to proving that, for any $S\in \mathcal{S}(\cC_2 \times \cC_2)$: 
\vspace{1.8mm}

\noindent(1) If $\cC_2\times\cC_{12}\in S$, then $S=[\cC_2\times\cC_{12}]$: Suppose that there exists another $H\in \Phi_\Q(2,\cC_2\times\cC_2)$ such that $H\in S$. Then there exists an elliptic curve defined over $\Q$ and two squarefree integers $D,D'$ such that $E(\Q(\sqrt{D}))_{\tors}=\cC_2\times\cC_{12}$ and $E(\Q(\sqrt{D'}))_{\tors}=H$.
\begin{itemize}
\item Suppose that $H=\cC_2\times\cC_4$. Then there is a point of order $12$ and a point of order $4$ in different fields, and therefore they generate different rational points of order $4$. That implies we may have $\cC_4\times\cC_{12}$ over the biquadratic field $\Q(\sqrt{D},\sqrt{D'})$, but Proposition \ref{ClEx} tells us that this group can only appear at degree $2^3$ or larger.
\item  Suppose that $H=\cC_2\times\cC_6$.  Then we would have $\cC_6\times\cC_6 \leq E(\Q(2^\infty))_{\tors}$. This contradicts Theorem \ref{FAA1}.
\item Finally, assume that $H=\cC_2\times\cC_8$. Then $\cC_8\times\cC_{12} \leq E(\Q(2^\infty))_{\tors}$. This again contradicts Theorem \ref{FAA1}.
\end{itemize}

\noindent (2) $\left[\cC_2\times\cC_{6}\,,\,\cC_2\times\cC_{8}\right]\not\subset S$. Were this the case we would have $\cC_6\times\cC_{8} \le E(\Q(2^\infty))_{\tors}$ which is not possible (Theorem \ref{FAA1}).\\

\noindent (3) $S\ne \left[ \cC_2\times\cC_{6}, \, \cC_2\times\cC_{4}, \, \cC_2\times\cC_{4} \right]$. We will not give full details here, as they are similar to those in the previous subsection.

Let $E$ be an elliptic curve defined over $\Q$ such that $E(\Q)_{\tors}=\cC_2\times \cC_2 $ and there exist three squarefree integers $D_1,D_2,D$ such that 
\begin{eqnarray*}
E (\Q ( \sqrt{D_i} ) )_{\tors} &=& \cC_2\times \cC_4 \mbox{ for $i=1,2$}, \\
E (\Q (\sqrt{D} ))_{\tors} &=& \cC_2\times \cC_6.
\end{eqnarray*}
We are going to prove that this is impossible. In other words, $ \cC_4\times \cC_{12} \leq E(L)_{\tors}$ is not possible for any triquadratic field $L$. This is equivalent to the same statement, but for the elliptic curve $E_D$, since $E$ and $E_D$ are isomorphic over $\Q(\sqrt{D})$. For this purpose, we are going to use the general curve with torsion $\cC_2 \times \cC_6$ by Kubert \cite{Kubert} in the form given by Elkies \cite{Elkies}:
$$
  E':  Y^2 = \left(X+t^2 \right)  \left(X+(t+1)^2 \right)  \left( X+(t^2+t)^2 \right)
$$
with $3$-torsion points at $X=0$.  Now mind that, if the curve  $Y^2 = X (X^2+aX+b)$ has a $4$-torsion point $T$ such that $2T = (0,0)$, then the  first coordinate of $T$ is a square root of $b$.  For $E'$, there are three choices of $b$, all equivalent. This is because, projectively, $E'$ can be written as
$$
  Y^2 = \left( X+(tu)^2 \right) \left(X+(tv)^2 \right) \left( X+(uv)^2 \right)
$$
with $t+u+v$ = 0. In our case the three possible $b$'s are:
$$
t^3 (2 + t) (1 + 2 t), \; -(-1 + t) (1 + t)^3 (1 + 2 t), \; (-1 + t) t^3 (1 + t)^3 (2 + t).
$$
Once $E_D$ has full $4$--torsion over some number field $L$ then $L$ must contain $\sqrt{-1}$ from Proposition \ref{nthroot}; so there are really only two other square roots that one needs to specify to determine the triquadratic field. If two of the $b$'s yield points defined over the same quadratic field then either one of these $b$'s is a square or two of them multiply to a square.  But this is already enough because each possibility yields an elliptic curve of rank zero (\texttt{24a1} and \texttt{48a1}) and the torsion points on both curves correspond to singular curves in the equation $E'$.

\noindent (4) If $\left[\cC_2\times\cC_{4},\cC_2\times\cC_{4},\cC_2\times\cC_{4}\right]\subset S$, then $S=\left[\cC_2\times\cC_{4},\cC_2\times\cC_{4},\cC_2\times\cC_{4}\right]$. A group $\cC_2\times\cC_{6}$ cannot appear in $S$ from the argument above. And $\cC_2\times\cC_{8}$ cannot appear either because there would be a point of order $8$ in a quadratic extension, coming from halving a point of order $4$, but we have already obtained all possible quadratic extension where the torsion grows ($3$, in fact, from Proposition \ref{2dmnc}). 

All the remaining cases do happen, as shown in Table \ref{tablagrande}. Therefore we have proved:
\begin{eqnarray*}
\mathcal H_\Q(2,\cC_2\times\cC_2)  &=& \Big\{ \left[ \cC_2 \times \cC_4 \right]; \; \left[ \cC_2 \times \cC_6 \right]; \; 
\left[ \cC_2 \times \cC_8 \right]; \; \left[ \cC_2 \times \cC_{12} \right]; \\
&& \quad \left[ \cC_2 \times \cC_4, \cC_2 \times \cC_4 \right]; \; \left[ \cC_2 \times \cC_4, \cC_2 \times \cC_6 \right]; \; \left[ \cC_2 \times \cC_4, \cC_2 \times \cC_8 \right]; \\
&& \quad \left[ \cC_2 \times \cC_4, \cC_2 \times \cC_4,\cC_2 \times \cC_4 \right]; \; \left[ \cC_2 \times \cC_4, \cC_2 \times \cC_4,\cC_2 \times \cC_8 \right] \Big\}.
\end{eqnarray*}

\subsection{The group $\cC_2$}

\

Some quick remarks on $\mathcal H_\Q(2,\cC_2)$ beforehand:

First, no element of $\mathcal H_\Q(2,\cC_2)$ can contain both $\cC_{10}$ (or $\cC_2\times\cC_{10}$) and $\cC_m$ with some $m\ge 4$. The reason for this is that no element in $\Phi_\Q(2^\infty)$ has points of order $10$ and points of order $m$.
This, together with the remark at the beginning of the section, shows that:

\begin{itemize}
\item $\cC_2\times\cC_{10}$ can only appear in an element of $\mathcal H_\Q(2,\cC_2)$ as $\left[ \cC_2\times\cC_{10}\right]$.
\item $\cC_{10}$ can only appear as $\left[ \cC_{10}, \, \cC_2\times\cC_{2} \right]$.
\end{itemize}

Second, there are some pairs which cannot appear together in an element of $\mathcal H_\Q(2,\cC_2)$:
\begin{itemize}
\item $\cC_{6}$ (or $\cC_2\times\cC_{6}$) and $\cC_8$, as there is no $H\in \Phi_\Q(2^\infty)$ with points of order $6$ and points of order $8$.
\item $\cC_{8}$ and $\cC_{16}$. Assume $\cC_8 = \langle P \rangle$ and $\cC_{16} = \langle Q \rangle$ are the torsion subgroups in two different quadratic extensions. Consider the group homomorphism
\begin{eqnarray*}
\varphi: \cC_8 \times \cC_{16} & \longrightarrow & E(\Q(2^\infty)) \\
(nP,mQ) & \longmapsto & nP + mQ
\end{eqnarray*}
which verifies $\ker(\varphi) = \langle (4P,8Q) \rangle$, as the rational point of order $2$ is the only one who has its inverse in both quadratic extensions. 

So $E(\Q(2^\infty))_{\tors}$ contains a group of $64$ elements with (at least) an element of order $8$ and no elements of order $16$. From Theorem \ref{FJNT1} this would imply there exists an elliptic curve $E$ defined over $\Q$ such that $E(\Q)_{\tors}=\cC_2$ and $\cC_8\times\cC_8 \leq E(\Q(2^\infty))_{\tors}$ and this contradicts Proposition \ref{FJNT2}.
\end{itemize}

Another important remark here is the following: let $E$ be an elliptic curve defined over $\Q$ such that there is a quadratic extension $K/\Q$ with $\cC_n = E(K)_{\tors}$, and $4|n$, then there must be another quadratic extension $K'/\Q$ with $\cC_m = E(K')_{\tors}$ with $4|m$. Moreover, there are no more extensions where the torsion grows, apart from the splitting field of $X^3+AX+B$ which gives a non--cyclic torsion group. This can be deduced from Lemma \ref{2dm2} as there are either $2$ or no quadratic extension where one can get points of order $4$ and, therefore, groups $\cC_n$ and $\cC_m$ with $n,m \in 4\Z$. The following pairs may then appear:
$$
\{\cC_{4},\cC_{4}\},\; \{\cC_{4},\cC_{8}\},\; \{\cC_{4},\cC_{12}\},\; \{\cC_{4},\cC_{16}\},\;
\{\cC_{8},\cC_{8}\},\; \{\cC_{8},\cC_{12}\},\; \{\cC_{8},\cC_{16}\},\; \{\cC_{12},\cC_{16}\},
$$
although the last three ones can already be ruled out from the arguments above.

\vspace{1mm}
Let us then construct the elements $S \in \mathcal H_\Q(2,\cC_2)$ in ascending order of $\#S$:
\begin{itemize}
\item $\#S=1$: In this case $S\in\{[\cC_2\times\cC_{2}],\; [\cC_2\times\cC_{6}],\; [\cC_2\times\cC_{10}]\}$. All of these cases can occur (see examples in Table \ref{tablagrande}).
\item $\#S=2$: In Table \ref{tablagrande} we can find examples of: 
$$
[\cC_2\times\cC_{2}, \;\cC_{6}], \; [\cC_2\times\cC_{2},\;\cC_{10}], \; [\cC_2\times\cC_{6},\;\cC_{6}].
$$
These are all the possibilities, from Theorem \ref{teo1} and the previous remarks.
\item $\#S=3$ with $\cC_2\times\cC_{2}\in S$. We have example for all the possible cases (after taking into account the preliminary remarks), which are:
$$
[\cC_2\times\cC_{2},\;\cC_{4},\;\cC_{4}], \;\; [\cC_2\times\cC_{2},\;\cC_{4},\;\cC_{8}], \;\; [\cC_2\times\cC_{2},\;\cC_{4},\;\cC_{12}],
$$
$$
[\cC_2\times\cC_{2},\;\cC_{4},\;\cC_{16}], \;\; [\cC_2\times\cC_{2},\;\cC_{8},\;\cC_{8}], \;\;
[\cC_2\times\cC_{2},\;\cC_{6},\;\cC_{6}].
$$
\item $\#S=3$ with $\cC_2\times\cC_{6}\in S$. We have examples for $[\cC_2\times\cC_{6},\;\cC_{4},\;\cC_{4}]$ and the rest can be ruled out. Precisely:
$$
[\cC_2\times\cC_{6},\;\cC_{4},\;\cC_{8}], \;\; [\cC_2\times\cC_{6},\;\cC_{4},\;\cC_{16}], \;\; [\cC_2\times\cC_{6},\;\cC_{8},\;\cC_{8}]
$$
cannot appear because there is no $H\in \Phi_\Q(2^\infty)$ with points of order $6$ and points of order $8$. Also
$$
[\cC_2\times\cC_{6},\;\cC_{4},\;\cC_{12}]
$$
is not an option, as that would imply $\cC_3\times\cC_{12}$ is a subgrup of some $H\in\Phi_\Q(2^\infty)$. Finally,
$$
[\cC_2\times\cC_{6},\;\cC_{6},\;\cC_{6}]
$$
is not an option. Were this the case, we would have three $\cC_3$ subgroups (different pairwise, as they appear in different quadratic extensions) of some $H\in \Phi_\Q(2^\infty)$, which is not possible.
\item $\#S=4$ with $\cC_2\times\cC_{2}\in S$. We have examples (see Table \ref{tablagrande} as usual) for
$$
S=[\cC_2\times\cC_{2},\;\cC_{4},\;\cC_{4},\;\cC_{6}],
$$
and the remaining possibilities do not happen, in a similar way as the previous case. In fact,
$$
[\cC_2\times\cC_{2},\;\cC_{4},\;\cC_{8},\;\cC_{6}], \;\; [\cC_2\times\cC_{2},\;\cC_{4},\;\cC_{16},\;\cC_{6}], \;\; 
[\cC_2\times\cC_{2},\;\cC_{8},\;\cC_{8},\;\cC_{6}]
$$
all have points of order $6$ and points of order $8$, while 
$$
[\cC_2\times\cC_{2},\;\cC_{4},\;\cC_{12},\;\cC_{6}],
$$
would imply $\cC_3\times\cC_{12} \leq H$ for some group $H\in \Phi_\Q(2^\infty)$. 

\item $\#S=4$ with $\cC_2\times\cC_{6}\in S$. The only case would be $S=[\cC_2\times\cC_{6},\;\cC_{4},\;\cC_{4},\;\cC_{6}]$ and in fact it does not occur, as it would imply $\cC_3\times\cC_{12}$ is a subgroup for a certain $H\in\Phi_\Q(2^\infty)$.

\item $\#S=5$. The only possible case would be $S=[\cC_2\times\cC_{2},\;\cC_{4},\;\cC_{4},\;\cC_{6},\;\cC_{6}]$, which would imply, again,  $\cC_3\times\cC_{12} \leq H$, for some $H\in \Phi_\Q(2^\infty)$.
\end{itemize}
Therefore we have proved:

\begin{eqnarray*}
\mathcal H_\Q \left( 2,\cC_2\times\cC_2 \right) &=& \Big\{ \; \left[ \cC_2\times\cC_2 \right]; \, \left[ \cC_2\times\cC_6 \right]; \,
\left[ \cC_2\times\cC_{10} \right]; \, \left[ \cC_2\times\cC_2, \, \cC_6 \right]; \, \left[ \cC_2\times\cC_2 , \, \cC_{10} \right]; \\
&& \quad  \left[ \cC_2\times\cC_6, \, \cC_6 \right]; \; \left[ \cC_2\times\cC_2, \, \cC_4, \, \cC_4 \right];\, \left[ \cC_2\times\cC_2, \, \cC_6, \, \cC_6 \right];\\
&& \quad \left[ \cC_2\times\cC_2, \, \cC_8, \, \cC_8 \right]; \, \left[ \cC_2\times\cC_2, \, \cC_4, \, \cC_8 \right]; \, \left[ \cC_2\times\cC_2, \, \cC_4, \, \cC_{12} \right]; \\
&& \quad  \left[ \cC_2\times\cC_2, \, \cC_4, \, \cC_{16} \right]; \, \left[ \cC_2\times\cC_6, \, \cC_4, \, \cC_4 \right]; \, \left[ \cC_2\times\cC_2, \, \cC_4, \, \cC_4, \, \cC_6 \right] \Big\}.
\end{eqnarray*}

This finishes the proof of Theorem \ref{teo3}.

\

\section*{Appendix: Computations}

Let $G \in \Phi(1)$, $S= \left[ H_1,...,H_m \right]\in\mathcal{H}_{\Q}(2,G)$, $E$ an elliptic curve defined over $\Q$ such that $E({\mathbb Q})_{\tors} = G$ and let $D_1,\dots,D_m\in\mathbb Z$, squarefree, such that 
$$
E( \Q(\sqrt{D_i}))_{\tors} = H_i \mbox{ for } i=1,...,m.
$$
Let us write  
$$
F_S=\Q \left( \sqrt{D_1},\dots,\sqrt{D_m} \right).
$$
Table \ref{tablagrande} shows an example of every possible situation, where at
\begin{itemize}
\item the first column is $S$,
\item the second column is $S \in\mathcal{H}_{\Q}(2,G)$,
\item the third column is $\# S$,
\item the fourth column is $E(F_S)_{\tors}$,
\item the fifth column is the degree of $F_S$ over $\Q$,
\item the sixth column is the label of the elliptic curve $E$ with minimal conductor satisfying the conditions above, 
\item the seventh column displays the $D's$ corresponding to the respective $H's$ in $S$.
\end{itemize}

\noindent {\bf Remark.--} With the previous notation, we have computed for any curve in the Antwerp--Cremona tables \cite{cremonaweb}: $G$, $S$ and $E(F_S)_{\tors}$. 
Interestingly, for a given $S$, the group $E(F_S)_{\tors}$ seem to be fully determined, except for the cases
$$
G = \cC_2; \quad S = \left[ \cC_2 \times \cC_2, \; \cC_4, \; \cC_ 4 \right]; 
$$
$$
G = \cC_2 \times \cC_2; \quad S = \left[ \cC_2 \times \cC_4, \; \cC_2 \times \cC_ 4 \right]
$$
where two different $E(F_S)$ appear as we run through the entire set of curves in \cite{cremonaweb}. Given the amount of computations we have carried out, we think it is safe to conjecture that this is precisely the case

\

\noindent {\bf Remark.--} Comparing the results in Table \ref{tablagrande} with the set $\Phi_\Q \left( 2^\infty \right)$ we can conclude that the only groups in  $\Phi_\Q(2^\infty)$ which do {\em not} appear if we consider the groups $E( F_S )_{\tors}$ are:
$$
\cC_4\times\cC_{12}, \quad \cC_4\times\cC_{16}, \quad \cC_8\times\cC_{8}.
$$
These are, precisely, the groups discussed at Proposition \ref{ClEx}. Our computations suggest that this is in fact the case, but we have not proved this in detail.

\begin{table}[ht]
\caption{$h= \# S \mbox{ for } S \in \mathcal{H}_{\Q}(2,G)$, $d= [F_S:\Q]$}\label{tablagrande}
\begin{tabular}{|c|l|c||c|c||c|c|}
\hline
$G$ & $\mathcal{H}_{\Q}(2,G)$ & $h$& $E(F_S)_{\tors}$ & $d$ & label & $D's$\\
\hline\hline
\multirow{6}{*}{$\cC_1$} & $\cC_3$  & \multirow{4}{*}{$1$}& $\cC_3$ & \multirow{4}{*}{$2$} &\texttt{19a2}& $-3$\\
\cline{2-2}\cline{4-4}\cline{6-7}
& $\cC_5$    & &$\cC_5$ & &\texttt{75a2} & $5$ \\
\cline{2-2}\cline{4-4}\cline{6-7}
& $\cC_7$    & & $\cC_7$&& \texttt{208d1}& $-1$\\
\cline{2-2}\cline{4-4}\cline{6-7}
& $\cC_9$   & & $\cC_9$& & \texttt{54a2}& $-3$\\
\cline{2-3}\cline{4-7}
& $\cC_3,\cC_3$   & \multirow{2}{*}{$2$}& $\cC_3\times\cC_3$ &  \multirow{2}{*}{$4$} & \texttt{175b2} & $5,-15$\\
\cline{2-2}\cline{4-4}\cline{6-7}
& $\cC_3,\cC_5$   & & $\cC_{15}$ & & \texttt{50a4}& $-3,5$ \\
\hline
\hline
\multirow{14}{*}{$\cC_2$} & $\cC_2\times\cC_2$  & \multirow{3}{*}{$1$}&$\cC_2\times\cC_2$ &  \multirow{3}{*}{$2$} & \texttt{46a1} & $-23$\\
\cline{2-2}\cline{4-4}\cline{6-7}
&  $\cC_2\times\cC_6 $ & & $\cC_2\times\cC_6 $& & \texttt{36a3} & $-3$\\
\cline{2-2}\cline{4-4}\cline{6-7}
& $ \cC_2\times\cC_{10} $   & & $ \cC_2\times\cC_{10} $&& \texttt{450a3} & $-15$\\
\cline{2-3}\cline{4-5}\cline{6-7}
& $ \cC_2\times\cC_2, \cC_6 $   & \multirow{3}{*}{$2$}& $\cC_2\times\cC_6 $&  \multirow{10}{*}{$4$}& \texttt{14a3} & $-7,-3$\\
\cline{2-2}\cline{4-4}\cline{6-7}
& $\cC_2\times\cC_2 , \cC_{10}$   & &  $\cC_2\times\cC_{10} $ & &  \texttt{150b3} & $-15,5$ \\
\cline{2-2}\cline{4-4}\cline{6-7}
& $ \cC_2\times\cC_6, \cC_6 $   & & $\cC_6\times\cC_{6} $ & &  \texttt{98a3} & $-7, 21$\\
\cline{2-3}\cline{4-4}\cline{6-7}
& \multirow{2}{*}{$\cC_2\times\cC_2, \cC_4, \cC_4 $}   & \multirow{8}{*}{$3$}& $\cC_2\times\cC_{4} $& &  \texttt{15a5} & $ 5, -1,-5$  \\
\cline{4-4}\cline{6-7}
&   & & $\cC_4\times\cC_{4} $& &  \texttt{64a4} & $-1,2,-2$ \\
\cline{2-2}\cline{4-4}\cline{6-7}
& $ \cC_2\times\cC_2, \cC_8, \cC_8 $   & & $\cC_4\times\cC_{8} $ & &  \texttt{2880r6} & $-1,6,-6$\\
\cline{2-2}\cline{4-4}\cline{6-7}
& $ \cC_2\times\cC_2, \cC_4, \cC_8 $   & & $\cC_2\times\cC_{8} $ & &  \texttt{24a6} & $-2,2,-1 $ \\
\cline{2-2}\cline{4-4}\cline{6-7}
& $\cC_2\times\cC_2, \cC_4, \cC_{12} $   & & $\cC_2\times\cC_{12} $ & &  \texttt{30a3} & $ -15,5,-3$ \\
\cline{2-2}\cline{4-4}\cline{6-7}
& $ \cC_2\times\cC_2, \cC_4,\cC_{16} $   & & $\cC_2\times\cC_{16} $& &  \texttt{3150bk1} & $-7, 105, -15$\\
\cline{2-2}\cline{4-4}\cline{6-7}
& $ \cC_2\times\cC_6, \cC_4, \cC_4 $   & & $\cC_2\times\cC_{12} $ & &  \texttt{450g1} & $-15,-3,5 $\\
\cline{2-2}\cline{4-5}\cline{6-7}
& $ \cC_2\times\cC_2, \cC_6, \cC_6 $   & &$\cC_6\times\cC_{6} $  & \multirow{2}{*}{$8$} &  \texttt{98a4} & $2, -7, 21$\\
\cline{2-3}\cline{4-4}\cline{6-7}
& $ \cC_2\times\cC_2, \cC_4, \cC_4, \cC_6$   & $4$& $\cC_2\times\cC_{12} $&  &  \texttt{30a7} & $ 10,-5, -2, -3$\\
\cline{2-2}
\hline\hline
\multirow{2}{*}{$\cC_3$}  & $ \cC_{15}$   &\multirow{2}{*}{$1$} & $ \cC_{15}$ & \multirow{2}{*}{$2$} &  \texttt{50a3} & $5$\\
\cline{2-2}\cline{4-4}\cline{6-7}
& $  \cC_3\times\cC_3 $  & & $  \cC_3\times\cC_3 $ & &  \texttt{19a1} & $-3$\\
\hline\hline
\multirow{7}{*}{$\cC_4$}  & $ \cC_2\times\cC_4$   & \multirow{4}{*}{$1$}& $ \cC_2\times\cC_4$& \multirow{4}{*}{$2$}&  \texttt{17a1} & $-1$ \\
\cline{2-2}\cline{4-4}\cline{6-7}
& $ \cC_2\times\cC_8 $   & & $ \cC_2\times\cC_8 $& &  \texttt{192c6} & $-2$ \\
\cline{2-2}\cline{4-4}\cline{6-7}
& $ \cC_2\times\cC_{12} $   & & $ \cC_2\times\cC_{12} $& &  \texttt{150c3} & $-15$ \\
\cline{2-2}\cline{4-4}\cline{6-7}
& $ \cC_4\times\cC_4 $   & & $ \cC_4\times\cC_4 $& &  \texttt{40a4} & $-1$ \\
\cline{2-3}\cline{4-7}\cline{6-7}
& $ \cC_2\times\cC_4,\cC_{12} $   & \multirow{3}{*}{$2$}& $ \cC_2\times\cC_{12} $ & \multirow{3}{*}{$4$} &  \texttt{90c1} & $-15,-3$\\
\cline{2-2}\cline{4-4}\cline{6-7}
& $ \cC_2\times\cC_4, \cC_8, \cC_8 $   & & $ \cC_2\times\cC_8 $& &  \texttt{15a7} & $15, 3,5$ \\
\cline{2-2}\cline{4-4}\cline{6-7}
&   $ \cC_2\times\cC_8, \cC_8, \cC_8 $ & & $ \cC_4\times\cC_8 $& &  \texttt{240d6} & $-1,6, -6 $\\
\hline\hline
$\cC_5$ & $ \cC_{15}$   & $1$& $ \cC_{15}$& $2$ &  \texttt{50b1} & $5$\\
\hline\hline
\multirow{3}{*}{$\cC_6$}  & $ \cC_2\times\cC_6 $   & $1$& $ \cC_2\times\cC_6 $& $2$ &  \texttt{14a4} & $-7$ \\
\cline{2-3}\cline{4-5} \cline{6-7}
& $ \cC_2\times\cC_6, \cC_3\times\cC_6 $   & $2$& $ \cC_6\times\cC_6 $ &  \multirow{2}{*}{$4$} &  \texttt{14a1} & $-7,-3$\\
\cline{2-3}\cline{4-4}\cline{6-7}
& $\cC_2\times\cC_6, \cC_{12}, \cC_{12}$   & $3$& $ \cC_2\times\cC_{12} $& &  \texttt{30a1} & $-15,-3, 5$\\
\cline{2-2}\cline{4-4}\cline{6-7}
\hline\hline
\multirow{2}{*}{$\cC_8$}  & $ \cC_2\times\cC_8 $   &  $1$&$ \cC_2\times\cC_8 $ & $2$ &  \texttt{15a4} & $-1$\\
\cline{2-3}\cline{4-5}\cline{6-7}
& $\cC_2\times\cC_8, \cC_{16}, \cC_{16}$   & $3$ & $\cC_2\times\cC_{16}$&  $4$ &  \texttt{210e1} & $-7, 105, -15$ \\
\hline\hline
$\cC_{10}$ & $ \cC_2\times\cC_{10}$   & $1$& $ \cC_2\times\cC_{10}$& $2$ &  \texttt{66c1} & $33$\\
\hline\hline
$\cC_{12}$ & $ \cC_2\times\cC_{12}$   & $1$& $ \cC_2\times\cC_{12}$& $2$ &  \texttt{90c3} & $-15$ \\
\hline
\hline
\multirow{9}{*}{$\cC_2\times\cC_2$}  & $ \cC_2\times\cC_4 $   &  \multirow{4}{*}{$1$}& $ \cC_2\times\cC_4 $& \multirow{4}{*}{$2$} &  \texttt{33a1} & $-11$ \\
\cline{2-2}\cline{4-4}\cline{6-7}
& $ \cC_2\times\cC_6 $   & &$ \cC_2\times\cC_6 $ & &  \texttt{30a6} & $-3$ \\
\cline{2-2}\cline{4-4}\cline{6-7}
& $ \cC_2\times\cC_8$   & & $ \cC_2\times\cC_8 $& &  \texttt{63a2} & $-3$\\
\cline{2-2}\cline{4-4}\cline{6-7}
& $ \cC_2\times\cC_{12} $   & & $ \cC_2\times\cC_{12} $& &  \texttt{960o6} & $6$\\
\cline{2-3}\cline{4-5}\cline{6-7}
&  \multirow{2}{*}{$ \cC_2\times\cC_4, \cC_2\times\cC_4 $}   &  \multirow{4}{*}{$2$}& $\cC_4\times\cC_4 $&  \multirow{6}{*}{$4$} &  \texttt{17a2} & $17, -1$\\
\cline{4-4}\cline{6-7}
&     & & $\cC_4\times\cC_8 $& &  \texttt{1200j4} & $-5,5$\\
\cline{2-2}\cline{4-4}\cline{6-7}
& $ \cC_2\times\cC_4, \cC_2\times\cC_6 $   & & $\cC_2\times\cC_{12} $& &  \texttt{90c2} & $6,-3$ \\
\cline{2-2}\cline{4-4}\cline{6-7}
& $ \cC_2\times\cC_4, \cC_2\times\cC_8 $   & & $\cC_4\times\cC_8 $& &  \texttt{75b3} & $-5,5$\\
\cline{2-3}\cline{4-4}\cline{6-7}
& $ \cC_2\times\cC_4, \cC_2\times\cC_4, \cC_2\times\cC_4 $   &  \multirow{2}{*}{$3$}& $ \cC_4\times\cC_4 $  & &  \texttt{15a2} & $-5, 5, -1$\\
\cline{2-2}\cline{4-4}\cline{6-7}
& $ \cC_2\times\cC_4, \cC_2\times\cC_4, \cC_2\times\cC_8 $   & & $ \cC_4\times\cC_8 $ & &  \texttt{510e5} & $-34, 34, -1$ \\
\hline\hline
\multirow{5}{*}{$\cC_2\times\cC_4$} & $ \cC_2\times\cC_8 $   &\multirow{2}{*}{$1$} &$ \cC_2\times\cC_8 $  & \multirow{2}{*}{$2$} &  \texttt{15a3} & $5$ \\
\cline{2-2}\cline{4-4}\cline{6-7}
& $ \cC_4\times\cC_4 $   & &$ \cC_4\times\cC_4 $ & &  \texttt{195a3} & $-1$\\
\cline{2-3}\cline{4-5}\cline{6-7}
& $ \cC_2\times\cC_8, \cC_4\times\cC_4 $   & \multirow{2}{*}{$2$}& $ \cC_4\times\cC_8 $ &\multirow{3}{*}{$4$} &  \texttt{15a1} & $5, -1$\\
\cline{2-2}\cline{4-4}\cline{6-7}
& $ \cC_2\times\cC_8, \cC_2\times\cC_8$   & & $ \cC_4\times\cC_8 $ & &  \texttt{1230f2} & $41, -1$\\
\cline{2-3}\cline{4-4}\cline{6-7}
& $ \cC_2\times\cC_8, \cC_2\times\cC_8, \cC_4\times\cC_4 $   & $3$& $ \cC_4\times\cC_8 $ & &  \texttt{210e3} & $-6, 6, -1$\\
\hline\hline
$\cC_2\times\cC_6$ & $ \cC_2\times\cC_{12}$   & $1$& $ \cC_2\times\cC_{12}$& $2$ &  \texttt{90c6} & $6$\\
\hline\end{tabular}
\end{table}


\begin{thebibliography}{11}

\bibitem{antwerp}
Birch, B.J.; Kuyk, W. (eds.): Modular Functions of One Variable IV. Lecture Notes in Mathematics {\bf 476}. Springer (1975).

\bibitem{cremonaweb}
Cremona, J.E.: {\em Elliptic curve data for conductors up to 300.000.} Available on \verb|http://www.warwick.ac.uk/~masgaj/ftp/data/|, 2013.

\bibitem{Elkies}
Elkies, N. D.:{\em Wiles minus epsilon implies Fermat.} Elliptic curves, modular forms, \& Fermat's last theorem (Hong Kong, 1993), 38--40, Ser. Number Theory, I, Int. Press, Cambridge, MA, 1995. 


\bibitem{Fujita2004}
Fujita, Y.: {\em Torsion subgroups of elliptic curves with non--cyclic torsion over $\Q$ in elementary abelian $2$--extensions of $\Q$}. Acta Arith.{\bf 115} (2004) 29--45.

\bibitem{Fujita2005}
Fujita, Y.: {\em Torsion subgroups of elliptic curves in elementary abelian $2$--extensions of $\Q$}. J. Number Theory {\bf 114} (2005) 124--134.

\bibitem{Gonzalez-Jimenez2014}
Gonz\'alez--Jim\'enez, E.: {\em Covering techniques and rational points on some genus $5$ curves.} To appear in Contemporary Mathematics AMS.

\bibitem{Gonzalez-Jimenez-Tornero2014}
Gonz\'alez--Jim\'enez, E.; Tornero, J.M.: {\em Torsion of rational elliptic curves over quadratic fields.} Rev. R. Acad. Cienc. Exactas Fís. Nat. Ser. A Math. RACSAM {\bf 108} (2014), 923--934.

\bibitem{Husemoller2004}
Husemoller, D.: Elliptic curves. Graduate Texts in Mathematics {\bf 111}. Springer (2004).

\bibitem{Jeon-Kim-Lee2013}
Jeon, D.; Kim, C.H.; Lee, Y.: {\em Infinite families of elliptic curves over dihedral quartic number fields.} J. Number Theory {\bf 133 } (2103) 115--122.

\bibitem{Kamienny1992}
Kamienny, S.: {\em Torsion points on elliptic curves and $q$--coefficients of modular forms}. Invent. Math. {\bf 109} (1992) 129--133.

\bibitem{Kenku-Momose1988}
Kenku, M.A.; Momose, F.: {\em Torsion points on elliptic curves defined over quadratic fields}. Nagoya Math. J. {\bf 109} (1988) 125--149.

\bibitem{Knapp1992}
Knapp, A.W.: Elliptic curves. Mathematical Notes, {\bf 40}. Princeton University Press (1992).

\bibitem{Kubert}
Kubert, D.S.: {\em Universal bounds on the torsion of elliptic curves.} Proc. London Math. Soc. {\bf 33} (1976) 193--237.

\bibitem{Kwon1997}
Kwon, S.: {\em Torsion subgroups of elliptic curves over quadratic extensions}. J. Number Theory {\bf 62} (1997) 144--162.

\bibitem{Laska-Lorenz1985}
Laska, M.; Lorenz, M.: {\em Rational points on elliptic curves over $\Q$ in elementary abelian $2$--extensions of $\Q$.} J. Reine Angew. Math. {\bf 355} (1985) 163--172.

\bibitem{Mazur1977}
Mazur, B.: {\em Modular curves and the Eisenstein ideal}. Publ. Math. Inst. Hautes \'Etudes. Sci. {\bf 47} (1977) 33--186.

\bibitem{Mazur1978}
Mazur, B.: {\em Rational isogenies of prime degree}. Invent. Math. {\bf 44} (1978) 129--162.

\bibitem{Najman2012preprint}
Najman, F.: {\em Torsion of elliptic curves over cubic fields and sporadic points on $X_1(n)$}. Math. Res. Lett., to appear.

\bibitem{Najman2014preprint}
Najman, F.: {\em The number of twists with large torsion of an elliptic curve}. Rev. R. Acad. Cienc. Exactas Fís. Nat. Ser. A Math. RACSAM., to appear.

\bibitem{Silverman2009}
Silverman, J.H.: {\em The arithmetic of elliptic curves}. Graduate Texts in Mathematics {\bf 106}. Springer (2009).
\end{thebibliography}

\end{document}